\documentclass[11pt,a4paper]{article}

\usepackage[active]{srcltx}
\usepackage[indent]{parskip}
\usepackage[title]{appendix}
\usepackage[margin = 1.75cm]{geometry}

\usepackage{todonotes}

\usepackage{mdframed}
\usepackage[normalem]{ulem}

\usepackage{amsmath,amssymb,amsthm,mathrsfs}
\usepackage{bbm,bm,float,url}

\allowdisplaybreaks

\usepackage{graphicx,subcaption,natbib}
\usepackage[colorlinks=true, allcolors=blue]{hyperref}

\usepackage{booktabs}
\usepackage{multirow}
\usepackage{cleveref}
\usepackage{doi}

\newtheorem{theorem}{Theorem}[section]
\newtheorem{corollary}[theorem]{Corollary}
\newtheorem{lemma}[theorem]{Lemma}
\newtheorem{proposition}[theorem]{Proposition}
\newtheorem{definition}[theorem]{Definition}

\newtheorem{assumption}[theorem]{Assumption}

\newtheorem{remark}[theorem]{Remark}

\newcommand{\1}{\mathbbm{1}}
\newcommand{\st}{\mathrm{s.t.}}
\newcommand{\R}{\mathbb{R}}
\newcommand{\Q}{\boldsymbol{Q}}
\newcommand{\B}{\mathcal{B}}
\newcommand{\Sumi}{\sum_{i=1}^N}
\newcommand{\Sumj}{\sum_{j=1}^J}
\newcommand{\Suml}{\sum_{\ell=1}^L}
\newcommand{\indexi}{i{\in}[1{:}N]}
\newcommand{\indexj}{j{\in}[1{:}J]}
\newcommand{\indexk}{k{\in}[1{:}K]}
\newcommand{\indexl}{\ell{\in}[1{:}L]}
\title{{\Large {\sf Convergent Lifted Lasserre Hierarchy of SDPs for Minimizing Expectation of Piecewise Polynomial Loss over Wasserstein Balls   }}}

\author{ Neil De Vera Dizon\thanks{Corresponding Author: Department of Applied Mathematics, University of New South Wales, Sydney 2052, Australia. Email: \url{n.dizon@unsw.edu.au}. Research of Dr. Neil Dizon was supported by a grant from the Australian Research Council.}, \ Queenie Yingkun Huang\thanks{Department of Applied Mathematics, University of New South Wales, Sydney 2052, Australia. Email: \url{yingkun.huang@unsw.edu.au}. The research of Ms. Queenie Yingkun Huang was partially supported by a grant from UNSW Sydney.}, \ Thai Doan Chuong\thanks{%
Department of Mathematics, Brunel University of London, London, England. Email: \url{chuongthaidoan@yahoo.com}; \url{chuong.thaidoan@brunel.ac.uk}.  }, \\ Guoyin Li\thanks{Department of Applied Mathematics, University of New South Wales, Sydney 2052, Australia. Email: \url{g.li@unsw.edu.au}. The research of Prof. Guoyin Li was supported by a grant from the Australian Research Council. } \ and \ Vaithilingam Jeyakumar\thanks{Department of Applied Mathematics, University of New South Wales, Sydney 2052, Australia. Email: \url{v.jeyakumar@unsw.edu.au}. The research of Prof. Vaithilingam Jeyakumar was supported by a grant from the Australian Research Council.}}
\begin{document}
\date{{\bf Revisied Version}: \today}
\maketitle

\begin{abstract} This paper investigates the minimization of the expectation of piecewise polynomial loss functions over Wasserstein balls. This optimization problem often appears as a key sub-problem of distributionally robust optimization problems. We establish the asymptotic convergence of a hierarchy of semi-definite programming (SDP) relaxations, providing a framework for approximating the optimal values of these inherently infinite-dimensional optimization problems. A central foundational contribution is the development of a new lifted positivity certificate: we demonstrate that piecewise polynomials positive over Archimedean basic semi-algebraic sets admit a structured system of sum-of-squares (SOS) representations. Furthermore, we prove that the proposed hierarchy achieves finite convergence under suitable conditions when the defining polynomials are convex. The practical utility and versatility of this approach are demonstrated via numerical experiments in revenue estimation and portfolio optimization.
\end{abstract}


%%%%%%%%%%%%%%%%%%%%%%%%%%%%%%%%%%%%%%%%%%%%
\section{Introduction} \label{sec:introduction}

Consider the expectation problem:
\begin{align} \tag{EP} \label{problem:expectation-intro}
    \min_{\mathbb{Q} \in \mathbb{B}_\varepsilon(\widehat{\mathbb{P}} )} \mathbb{E}^{\mathbb{Q}} \Big [ \min_{k=1,\ldots,K}\  \max_{j=1,\ldots,J} g_j^{(k)} (\boldsymbol{\xi})\Big ]:=  \min_{\mathbb{Q} \in \mathbb{B}_\varepsilon(\widehat{\mathbb{P}} )}\int_{\Xi}  \min_{k=1,\ldots,K}\  \max_{j=1,\ldots,J} g_j^{(k)} (\boldsymbol{\xi}) \ \mathbb{Q}  (d\boldsymbol{\xi}),
\end{align}
where each $g_j^{(k)}$ is a polynomial and $g(\boldsymbol{\xi}) := \min_{k=1,\ldots,K}\  \max_{j=1,\ldots,J} g_j^{(k)} (\boldsymbol{\xi})$ is termed as a \textit{piecewise polynomial} loss. A piecewise polynomial is not necessarily a polynomial. Also, it is not necessarily a convex function even if $g_j^{(k)}$'s are all convex polynomials. 
The support set $\Xi = \{\boldsymbol{\xi} \in \R^m : h_\ell (\boldsymbol{\xi}) \geq 0,\ \ell = 1, \ldots, L\}$ is a basic semi-algebraic set. The set $\mathbb{B}_\varepsilon (\widehat{\mathbb{P}}) = \{\mathbb{Q} \in \mathcal{P} (\Xi) : d_W (\mathbb{Q}, \widehat{\mathbb{P}}) \leq \varepsilon \}$ is the Wasserstein ball with %$\widehat{\mathbb{P}} = \frac{1}{N} \Sumi \1_{\widehat{\boldsymbol{\xi}}_i}$
$\widehat{\mathbb{P}}$ the empirical distribution of past realizations $\widehat{\boldsymbol{\xi}}_i \in \Xi$, $i=1, \ldots, N$ obtained from datasets.

Recent years have seen the emergence of the distributionally robust optimization (DRO) paradigm as a powerful framework for decision-making under uncertainty. At its core, the objective of DRO is to safeguard decisions against a range of potential probability distributions contained within a defined ambiguity set. \textit{A fundamental component of this framework is the expectation problem, which serves as a key subproblem in many DRO formulations \cite{mohajerin2018data, yue2022linear}}.

DRO offers distinct advantages over traditional methods. Unlike stochastic programming, which relies on a fixed distribution, DRO often demonstrates superior out-of-sample performance and enhanced computational tractability \cite{mohajerin2018data}. Simultaneously, it mitigates the inherent conservatism of deterministic robust optimization \cite{ben2009robust} by incorporating distributional information rather than focusing solely on worst-case scenarios.

\textbf{Motivation}. The piecewise polynomials appeared widely in DRO models as loss functions, although not all of the resulting DRO problems are numerically tractable. When $J=1$, the case where $g_1^{(k)}$'s are linear was considered in \cite{delage2010distributionally}, the case where $g_1^{(k)}$'s are quadratic functions was examined in \cite{kuhn2019wasserstein}, or more generally, the case where $g_1^{(k)}$'s are convex functions was studied in \cite{zhen2023unified,mohajerin2018data,han2015convex}. Moreover, when $J \geq 1$, $g_j^{(k)}$'s can be SOS-convex polynomials \cite{huang2024piecewise}, quadratic functions \cite{huang2025convexifiable}, or univariate polynomials \cite{popescu2005semidefinite}. 

Furthermore, decision-making in the face of uncertain conditions in a broad range of fields in finance, management, and economics can be modelled as an optimization problem of the form \eqref{problem:expectation-intro} with piecewise polynomials. For example, utility functions are frequently used in economics to quantify a person's risk preferences, and the piecewise power utility functions of order $c = 2,3,\ldots$, \cite{gerber1998utility} defined as: \begin{align*} u(\xi) =  \begin{cases}
    a (\xi - b)^c + d, & \text{if} \quad 0 \leq \xi < b, \\ d, & \text{if} \quad \xi \geq b,
\end{cases} \end{align*}where $a\in \R$, $b ,d > 0$, and $\xi$ may represent wealth or consumption, can be expressed as a piecewise polynomial. Then, minimizing the negative of the power utility function can be cast into the form of \eqref{problem:expectation-intro}. The piecewise representation of the cubic utility function can be found in \Cref{sec:rev_max}. 

In portfolio optimization, the piecewise utility function $\min_{k=1,\ldots,K} \{  a_k \boldsymbol{\xi}^\top \boldsymbol{x} + b_k\}$, where $\boldsymbol{\xi}$ and $\boldsymbol{x}$ represent uncertain investment returns and fixed portfolio allocation respectively, is a special case of the DRO problem \cite{delage2010distributionally}. In addition, the risk measure component of conditional value-at-risk (CVaR) \cite{rockafellar2000optimization} and the expected shortfall \cite{wozabal2012framework} are both piecewise polynomials. Both risk measures have been applied in portfolio optimization \cite{mohajerin2018data,wozabal2012framework}, and the CVaR is also applied in insurance management \cite{van2019distributionally} and adversarial classification \cite{ho2023adversarial}. 

\textbf{Approach and methodology}. 
Optimization problems of the form \eqref{problem:expectation-intro} are challenging infinite-dimensional optimization problems, and their numerical solvability depends on the structures of the loss function $g$, support set $\Xi$, and the ambiguity set $\mathscr{P}$. Various types of ambiguity sets have been studied in the DRO literature. The moment ambiguity sets are described by moment constraints using the mean and variance \cite{delage2010distributionally}. This approach is complementary to the probability distance-based ambiguity sets, which contain probability distributions that are ``close enough'' to a reference distribution \cite{gao2023distributionally}. 

A common approach for studying the problem \eqref{problem:expectation-intro} is to employ duality techniques to convert it into a semi-infinite program \cite{goberna2002linear,auslender2009penalty} that allows for the computation of solutions for problems with certain classes of loss functions of practical interest \cite{zhen2023unified,mohajerin2018data,wiesemann2014distributionally}.  Consequently, computationally tractable optimization reformulations and approximation schemes have been proposed for DRO problems \cite{delage2010distributionally,lasserre2008semidefinite,wiesemann2014distributionally,postek2016computationally,kuhn2024dro} in the sense that they can be expressed as efficiently solvable exact conic linear optimization problems or a convergent sequence of solvable conic linear problems.

\textbf{The Lasserre hierarchy of SDP relaxations}. Lasserre’s semi-definite program (SDP) hierarchies have been applied in DRO and moment problems. Notably, \cite{lasserre2008semidefinite} proposes a convergent hierarchy of SDPs to find the optimal value and optimal solution of a moment problem. Meanwhile, \cite{nie2023distributionally} presents a moment-sum-of-squares relaxation method for a class of moment-ambiguity DRO problems with a polynomial loss function. Also, \cite{klerk2020distributionally} outlines an SDP hierarchy for a DRO problem where the ambiguity set consists of probability distributions with sum-of-squares density functions. For a survey of SDP approaches to moment problems, readers are referred to \cite{klerk2019survey}. Further related results for min-max polynomial optimization can be found in \cite{bach2025sum,laraki2012semidefinite}.

In this paper, we will show that asymptotic convergence of a lifted Lasserre hierarchy of SDPs holds for \eqref{problem:expectation-intro}, which can be used to approximate the optimal values of the infinite-dimensional optimization problem \eqref{problem:expectation-intro}. This approach provides a computationally tractable method and algorithm for finding optimal values for various decision-making problems under uncertainty.

\textit{\textsf{Key contributions.}} Our key contributions are itemized below: 
\begin{itemize}
    
    \item We present a new lifted sum-of-squares representation of positivity for piecewise polynomials, establishing that piecewise polynomials positive over an Archimedean basic semi-algebraic set admit a system of sum-of-squares representations. 
    
    \item Building on this positivity representation, we propose a lifted hierarchy of semi-definite programs for the Wasserstein expectation problem of the form \eqref{problem:expectation-intro} with piecewise polynomials. By establishing asymptotic convergence, we show that the hierarchy of semi-definite program relaxations can approximate the optimal values of these infinite-dimensional expectation problems. 
    
    \item  Exploiting the structure of piecewise convexity, we further show that the SDP hierarchy of \eqref{problem:expectation-intro} achieves \textit{finite convergence}, under suitable conditions, when all the defining polynomials are convex. This shows that finite convergence of hierarchies for convex polynomial optimization continues to hold for expectation problems with piecewise convex polynomials that are not necessarily convex. It is done by invoking the convex-concave minimax theorem and Lagrange duality of convex analysis \cite{zalinescu2002convex} together with Scheiderer’s sum-of-squares representation \cite{scheiderer2003sums} of non-negative polynomials. 
    
    \item Finally, to illustrate the versatility of our approach, we apply it to concrete models of expectation problems \eqref{problem:expectation-intro} and distributionally robust decision-making models. We present results of a revenue estimation model with a piecewise polynomial utility function that is neither convex nor concave, exemplifying our SDP hierarchy. We also present results of numerical experiments, assessing the impact of increasing the relaxation level of our hierarchy on solution quality of a mean-CVaR portfolio optimization model.
\end{itemize}

\textit{\textsf{Originality.}} This work exploits the piecewise structure of polynomials to represent a wide range of practical piecewise functions and extends the recent advances in moment-based DRO, from the piecewise sum-of-squares-convex functions \cite{huang2024piecewise} and piecewise quadratic functions \cite{huang2025convexifiable} to Wasserstein DRO and general piecewise polynomials that are prevalent in a broad range of decision-making problems. A particular novelty of the present work lies in the derivation of a positivity certificate
showing that piecewise polynomials that are positive over Archimedean basic semi-algebraic sets admit a lifted system of sum-of-squares representations. By combining techniques from the positivstellensatz of algebraic geometry \cite{lasserre2009convexity} with the convex-concave minimax theorem of convex analysis \cite{zalinescu2002convex}, our approach enables the computation or estimation of the optimal values of otherwise infinite-dimensional expectation problems through SDPs, which can be efficiently solved using off-the-shelf software.

\textit{\textsf{Organization.}} The organization of the paper is as follows. We recall preliminary materials on the Wasserstein ball and positive polynomials in \Cref{sec:preliminary}. \Cref{sec:asymptotic-cvg} presents the positivity representation for piecewise polynomials, from which we derive an asymptotically convergent SDP hierarchy for the expectation problem \eqref{problem:expectation-intro}. \Cref{sec:finite-cvg} discusses a finitely convergent SDP hierarchy for \eqref{problem:expectation-intro}. The versatility of our approach is illustrated in \Cref{sec:dro} through concrete models of expectation and distributionally robust decision-making problems. Finally, \Cref{sec:conclusion} concludes with future research directions.

%%%%%%%%%%%%%%%%%%%%%%%%%%%%%%%%%%%%%%%%%%%%%%%%%%%%%%%%
\section{Preliminaries} \label{sec:preliminary}

This section presents the notation and definitions used in the paper, and preliminary details on the Wasserstein ambiguity set and positive polynomials on a compact basic semi-algebraic set. 

Denote $\R^m$ the Euclidean space of dimension $m$, $\R^m_+$ the non-negative orthant of $\R^m$, $\boldsymbol{a}^\top \boldsymbol{b}$, $\boldsymbol{a},\boldsymbol{b} \in \R^m$, the standard inner product on $\R^m$, and $\Delta = \{\boldsymbol{\delta} \in \R_+^{m} : \sum_{j=1}^m \delta_{j} = 1\}$ the simplex. Let $\boldsymbol{e}_j$ be the $j$\textsuperscript{th} standard Euclidean basis vector, $\boldsymbol{e}$ the vector of ones, and $\boldsymbol{0}$ the vector of zeros in appropriate dimensions. For integers $m,n$ with $m \le n$, we define $[m{:}n]:=\{m,  \ldots, n\}$. 

Denote $\mathbb{S}^m$ the space of $(m\times m)$ real symmetric matrices. A matrix $B \in \mathbb S^m$ is positive semi-definite, denoted as $B \succeq 0$ (resp. positive definite, denoted as $B \succ 0$), if $\boldsymbol{x}^\top B \boldsymbol{x} \geq 0$  for all $\boldsymbol{x}\in \R^m$ (resp. $\boldsymbol{x}^\top B \boldsymbol{x} >0$ for all $\boldsymbol{x} \in \R^m$, $\boldsymbol{x} \neq \boldsymbol{0}$). Let $\mathbb{S}_+^m$ be the cone of symmetric $(m \times m)$ positive semi-definite matrices.  We use lowercase letters for scalars, lowercase boldface letters for vectors, and capital letters for matrices. 

Let $X$ and $X'$ be real topological vector spaces. The spaces $X$ and $X'$ are paired if a bilinear form $\langle \cdot , \cdot \rangle : X' \times X \to \R$ is defined. The dual cone (positive polar cone) of a cone $S \subseteq X$ is defined as $S^* = \{x' \in X' : \langle x', x \rangle \geq 0,\ \forall x \in S \}$.  

Let $\Xi$ be a compact set in $\mathbb{R}^m$. Denote by $\B$ the Borel $\sigma$-algebra of $\Xi$.
Let $\mathcal{M}(\Xi)$ be the space of all (finite signed) regular Borel measures supported on $\Xi$. Let $C(\Xi)$ be the space of continuous functions supported on $\Xi$
equipped with the supremum norm. 

We use $\int_\Xi f ( \boldsymbol{\xi} )  \mathbb{P} (d \boldsymbol{\xi})$ refers to the expectation of the random variable $f(\boldsymbol{\xi})$ with respect to the measure $\mathbb{P}$ supported on $\Xi$. We also denote by $\mathcal{M}(\Xi)_+$ the convex cone of non-negative, finite, regular Borel measures supported on $\Xi$, that is, $\mathcal{M}(\Xi)_+ = \{\mu \in \mathcal{M}(\Xi) :  \mu \succeq_\B 0\}$, where $\mu \succeq_\B 0$ means that $\mu(A) = \mu(\xi \in A) \geq 0$ for all $\B$-measurable sets $A$. Let $\mathcal{P}(\Xi)$ be the set of all probability measures supported on $\Xi$. The Dirac measure $\1_{\boldsymbol{\xi}'}$ takes a mass of $1$ at the point $\boldsymbol{\xi}' \in \Xi$ and $0$ otherwise, so $\int_\Xi  g (\boldsymbol{\xi} ) \ \1_{\boldsymbol{\xi}'} (d \boldsymbol{\xi}) = g ( \boldsymbol{\xi}' )$.

% =============================================
\subsection{Wasserstein Ambiguity}\label{sec:Wasserstein-ambiguity}

Suppose that the empirical distribution is a discrete distribution given by $\widehat{\mathbb{P}} =\frac{1}{N} \Sumi \1_{\widehat{\boldsymbol{\xi}}_i}$, where $\widehat{\boldsymbol{\xi}}_i \in \Xi$, $\indexi$, are the past realizations obtained from datasets. Define the Wasserstein ambiguity set around the empirical distribution of radius $\varepsilon \geq 0$ as: \begin{eqnarray*}
    \mathbb{B}_\varepsilon(\widehat{\mathbb{P}} ):= \{\mathbb{Q} \in \mathcal{P}(\Xi): d_W (\mathbb{Q},\widehat{\mathbb{P}} ) \le \varepsilon\},
\end{eqnarray*}where $d_W$ denotes the (order 2) Wasserstein distance given by \begin{align*}
    d_W (\mathbb{Q},\mathbb{P})= \min _{\Pi \in \mathcal{P}(\Xi \times \Xi)} \left\{ \left(\int_{\Xi \times \Xi} \|\boldsymbol{\xi}-\boldsymbol{\xi}'\|_2^2 \ \Pi( d \boldsymbol{\xi}, d \boldsymbol{\xi}')\right)^{1/2}:\begin{matrix}
        \text{$\Pi$ is a joint distribution of $\boldsymbol{\xi}$ and $\boldsymbol{\xi}'$} \\ \text{with marginal distributions} \\ \text{$\mathbb{Q}$ and $\mathbb{P}$, respectively}
    \end{matrix}  \right\}
\end{align*}
with $\|\cdot\|_2$ the Euclidean norm in $\R^m$. This  Wasserstein distance is also called the quadratic Wasserstein distance (cf. \cite[Section~2.2]{pflug2014multistage}). 

The Wasserstein distance of two distributions $\mathbb{Q}$ and $\mathbb{P}$ can be viewed as the minimum transportation cost for moving the probability mass from $\mathbb{Q}$ to $\mathbb{P}$. The Wasserstein ambiguity set $\mathbb{B}_\varepsilon (\widehat{\mathbb{P}} )$ can be interpreted as the set of probability measures centred at the empirical distribution $\widehat{\mathbb{P}} $ with radius $\varepsilon \geq 0$. The larger the radius, the richer the ambiguity set. On the other hand, if the radius $\varepsilon$ shrinks to zero, the Wasserstein ball becomes the singleton set of the empirical distribution $\{\widehat{\mathbb{P}} \}$. 

Consider the key  expectation problem: \begin{align} \label{problem:mmt_0}
    \min_{\mathbb{Q} \in \mathbb{B}_\varepsilon(\widehat{\mathbb{P}} )} \mathbb{E}^{\mathbb{Q}} [g(\boldsymbol{\xi})],
\end{align}
where $g: \mathbb{R}^m \to \mathbb{R}$ is a continuous function. 

We associate with \eqref{problem:mmt_0} the following semi-infinite program:  \begin{align}\tag{ED} \label{problem:mmtd}
    \max_{\substack{ \lambda \geq 0, \alpha_i \in \R }} \left \{ -\lambda \varepsilon^2 + \frac{1}{N} \Sumi \alpha_i : \inf_{\boldsymbol{\xi} \in \Xi} \Big (g(\boldsymbol{\xi} ) + \lambda \| \boldsymbol{\xi} - \widehat{\boldsymbol{\xi}}_i \|_2^2  - \alpha_i \Big ) \geq 0, \ \forall \indexi \right \}. 
\end{align}

 Note that \eqref{problem:mmtd} is a semi-infinite optimization problem because the decision variable space $\R\times \R^N$ is finite-dimensional, but it can contain infinitely many non-negativity constraints $g(\boldsymbol{\xi}) + \lambda \| \boldsymbol{\xi} - \widehat{\boldsymbol{\xi}}_i \|_2^2 - \alpha_i \geq 0$, $\forall \boldsymbol{\xi} \in \Xi$, due to the (possibly) infinite set $\Xi \subseteq \R^m$.

The reduction from \eqref{problem:mmt_0} over the infinite-dimensional space of probability measures to the semi-infinite problem \eqref{problem:mmtd} is well-known in the literature. For example, \cite[Theorem~1]{blanchet2019quantifying} provides a strong duality result when the ambiguity set is defined through an optimal transport cost more general than the $\|\cdot \|_2^2$ in our setting. A strong duality result for a general Wasserstein distance is also noted in \cite[Theorem~1]{gao2023distributionally}. We provide the detailed proof in the appendix for the sake of completeness. 

\begin{lemma}[\textbf{Reduction to finite-dimensional program} \cite{gao2023distributionally}] \label{lemma:semi-inf}
    For the problem \eqref{problem:mmt_0}, let $\Xi \subseteq \R^m$ be a compact set, $g : \R^m \to \R$ be a continuous function, and $\varepsilon > 0$. Then, $\min \eqref{problem:mmt_0} = \max \eqref{problem:mmtd}$. 
\end{lemma}

\begin{proof}
    See Appendix \ref{appendix:finite-reduction}. 
\end{proof}

% ===================================
\subsection{Polynomials Positive on a Compact Basic Semi-algebraic Set}

We recall some basic definitions and results of sum-of-squares polynomials. Let $\R [\boldsymbol{\xi}]$ be the space of polynomials with real coefficients over $\boldsymbol{\xi} \in \R^m$. A polynomial $f \in \R[\boldsymbol{\xi}]$ is called a Sum-of-Squares (SOS) polynomial if there exist polynomials $f_i \in \R[\boldsymbol{\xi}]$, $i {\in} [1{:}n]$, such that $f = \sum_{i=1}^n f_i^2$. We use $\Sigma^2 (\boldsymbol{\xi})$ to denote the set of all SOS polynomials $f$ with respect to the variable $\boldsymbol{\xi}\in \R^m$. We also use $\Sigma_r^2(\boldsymbol{\xi})$ to denote the set of all SOS polynomials $f$ of degree at most $r$ with respect to the variable $\boldsymbol{\xi}\in \R^m$. The SOS property is a computationally tractable relaxation of the non-negativity of polynomials, and certifying whether a polynomial is SOS amounts to solving a semi-definite program \cite[Proposition~2.1]{lasserre2009moments}.

In what follows, we assume that the support set $\Xi\neq\emptyset$ is a compact basic semi-algebraic set, i.e., for $\{h_\ell\}_{\ell=1}^L \subset \R[\boldsymbol{\xi}]$,
\begin{equation}\label{def:supportset}
    \Xi = \{\boldsymbol{\xi} \in \R^m : h_\ell(\boldsymbol{\xi}) \geq 0,\,  \forall \indexl \}. 
\end{equation}

From now on, we fix the basic semi-algebraic set $\Xi$ as in \eqref{def:supportset}, and the cone 
\begin{equation}\label{eq:quadmod}
    \Q(h_1,\ldots,h_L) := \left\{f \in \R[\boldsymbol{\xi}]: f = \sigma_0 + \sum_{\ell=1}^{L} \sigma_\ell h_\ell, \ \sigma_\ell \in \Sigma^2(\boldsymbol{\xi}), \ \forall \ell{\in}[0{:}L] \right\}
\end{equation}
is the quadratic module generated by $\{h_\ell\}_{\ell=1}^L$.

We also recall the following Archimedean condition \cite[cf.][Theorem 2.16]{lasserre2015introduction} and Putinar’s Positivstellensatz (\cite{putinar1993positive}, \cite[cf.][Theorem 8]{klerk2019survey}), which is a representation result for polynomials that are strictly positive on a compact basic semi-algebraic set. 

\begin{assumption}[\textbf{Archimedean condition}] \label{assump:archimedean}
 Let $\{h_\ell\}_{\ell=1}^L \subset \R [\boldsymbol{\xi}]$, $\Xi$ defined in \eqref{def:supportset} be compact.  The Archimedean condition states that there exists $R\in \mathbb{N}$ such that $R - \| \boldsymbol{\xi} \|_2^2 \in \Q(h_1,\ldots,h_L)$.  
\end{assumption}

\begin{proposition}[\textbf{Putinar’s Positivstellensatz}]\label{prop:Putinar} Assume $\Xi$ is as defined in \eqref{def:supportset} and Assumption~\ref{assump:archimedean} holds. If $f \in \R[\boldsymbol{\xi}]$ is strictly positive on $\Xi$, then $f$ belongs to $\Q(h_1,\ldots,h_L)$.
\end{proposition}

Another key result that we will need is the following non-negativity representation by Scheiderer \cite[Example~3.18]{scheiderer2003sums}.

\begin{proposition}[\textbf{Scheiderer's non-negativity representation}]\label{prop:scheiderer}  Let $h\in \R[\boldsymbol{\xi}]$ be a polynomial for which the level set $\Omega = \{\boldsymbol{\xi} : h(\boldsymbol{\xi}) \ge 0\}$ is compact. Let $f (\boldsymbol{\xi})$ be non-negative on $\Omega$. Assume that the following conditions hold:
\begin{enumerate}
\item $f$ has only finitely many zeros in $\Omega$, each lying in the interior of $\Omega$, and
\item the Hessian of $f$ is positive definite at each of the zeros in $\Omega$. 
\end{enumerate}
Then, $f = \sigma_0 +  \sigma_1 \, h$ for some $\sigma_0,\sigma_1  \in \Sigma^2(\boldsymbol{\xi})$.
\end{proposition}

%%%%%%%%%%%%%%%%%%%%%%%%%%%%%%%%%%%%%%%%
\section{Expectation Problems: Asymptotically Converging SDPs } \label{sec:asymptotic-cvg}

This section presents a hierarchy of SDPs that approximates the expectation problem \eqref{problem:expectation-intro}, and its asymptotic convergence when the objective function involves a piecewise polynomial (see \Cref{defn:pieceise-poly}). This hierarchy is achieved by first establishing a sum-of-squares representation of a positive piecewise polynomial, which, in turn, guarantees the semi-infinite constraint in \eqref{problem:mmtd}. We begin by defining the notion of piecewise polynomials. 

\begin{definition}[Piecewise polynomial] \label{defn:pieceise-poly}
    A function $g : \R^m \to \R$ is called a piecewise polynomial if it can be expressed as \begin{align*}
        g(\boldsymbol{\xi}) = \min_{\indexk} \max_{\indexj} g_j^{(k)} (\boldsymbol{\xi}), \quad \boldsymbol{\xi} \in \R^m,
    \end{align*}for some polynomials $g_j^{(k)} \in \R[\boldsymbol{\xi}]$. 
\end{definition}

This form of piecewise polynomial appears in a variety of real-world application settings, and it is, in general, not a polynomial. It extends the piecewise sum-of-squares-convex (SOS-convex) function in \cite{huang2024piecewise} where each $g_j^{(k)}$ is a SOS-convex polynomial (see, e.g., \cite[Definition~2.4]{ahmadi2013complete}) or \cite[Section~1]{helton2010semidefinite}, minimax quadratic in \cite{huang2025convexifiable} where each $g_j^{(k)}$ is a quadratic function, and piecewise affine functions.

% =============================================
\subsection{Positivity Representation for Piecewise Polynomials}

The goal of this subsection is to derive a system of sum-of-squares representations for the positivity of a piecewise polynomial over a compact semi-algebraic set.

\begin{theorem}[\textbf{Lifted Putinar's positivstellensatz for piecewise polynomials}] \label{thm:positivity-representation}
Let $\Xi = \{\boldsymbol{\xi} \in \R^m : h_\ell (\boldsymbol{\xi}) \geq 0,\,  \forall \indexl \}$ as in \eqref{def:supportset}, and let the Archimedean condition in Assumption~\ref{assump:archimedean} hold for $\{h_\ell\}_{\ell=1}^L$. For each $\indexk$, let $g_j^{(k)} \in \R[\boldsymbol{\xi}]$, $\indexj$, and let $\tau^{(k)}_1$, $\tau_2^{(k)} \in \R$ be such that \begin{align} \label{eqn:tauk-conditions}
    \max_{\boldsymbol{\xi} \in 
    \Xi} \max_{\indexj}  g_j^{(k)}(\boldsymbol{\xi}) < \tau^{(k)}_1 \ \ \text{and}\ \ \tau^{(k)}_2 < \min_{\boldsymbol{\xi} \in \Xi} \min_{\indexj} g_j^{(k)} (\boldsymbol{\xi}) .
\end{align}

Let $f \in \R[\boldsymbol{\xi}]$. If \begin{align*}
    f (\boldsymbol{\xi}) +\min_{\indexk} \max_{\indexj} g_j^{(k)}(\boldsymbol{\xi}) > 0 \quad \text{for all} \quad \boldsymbol{\xi}\in \Xi, 
\end{align*}
then, for each $\indexk$, there exist $\eta_{j}^{(k)} \in \Sigma^2 (\boldsymbol{\xi},\widetilde{\xi})$, $j {\in} [1{:} (J {+}2)]$, $\sigma_\ell ^{(k)} \in \Sigma^2 (\boldsymbol{\xi},\widetilde{\xi})$, $\indexl$, such that \begin{align} \label{eqn:sos-rep2}
        \psi + f - \sum_{\ell=1}^L \sigma_\ell ^{(k)} h_\ell   - \sum_{j=1}^{J+2} \eta_{j}^{(k)}  \widetilde{g}_j^{(k)}  \in \Sigma^2(\boldsymbol{\xi},\widetilde{\xi}),
    \end{align}where $\psi (\boldsymbol{\xi}, \widetilde{\xi}) := \widetilde{\xi}$ and \begin{align*}
        \widetilde{g}_j^{(k)}(\boldsymbol{\xi},\widetilde{\xi}):= \begin{cases}
            \widetilde{\xi} - g_j^{(k)}(\boldsymbol{\xi}), \quad & \text{for} \quad \indexj, \\ 
            \tau^{(k)}_1 - \widetilde{\xi} \quad & \text{for} \quad j = J+1, \\ 
            \widetilde{\xi} - \tau^{(k)}_2, \quad & \text{for}\quad j = J+2. 
        \end{cases} 
    \end{align*}
\end{theorem}

\begin{proof} 
    \textit{This part follows from an application of Putinar's Positivstellensatz to a compact lifted set obtained by augmenting $\Xi$ with polynomial inequalities that encode upper and lower bounds on the functions $g_j^{(k)}$ over $\Xi$.}
    
    Fix $\indexk$, then,  
        \begin{align}\label{3.1-e-them}  
            \min_{\boldsymbol{\xi} \in \Xi} \Big \{ f ({\boldsymbol{\xi})} + \max_{\indexj} g_j^{(k)} (\boldsymbol{\xi}) \Big \}  > 0.
        \end{align}
    Under the Archimedean condition, $\Xi$ is compact and so $\tau^{(k)}_1$, $\tau^{(k)}_2$ in Eqn. \eqref{eqn:tauk-conditions} exist, and so,  
    {\small \begin{equation}\label{eqn:tau-ineq}
        \tau^{(k)}_1 - g_j^{(k)} (\boldsymbol{\xi}) > 0 \quad \text{and} \quad g_j^{(k)} (\boldsymbol{\xi}) - \tau^{(k)}_2 > 0 
    \end{equation}}for all $\boldsymbol{\xi} \in \Xi$ and for all $\indexj$. 
    
    Now, consider the set  
    \begin{align}\label{tilde-O-pos}
            \widetilde{\Xi}_k :=\{(\boldsymbol{\xi}, \widetilde{\xi}) \in \R^m \times \R \ : \  \widetilde{g}_j^{(k)} (\boldsymbol{\xi}, \widetilde{\xi}) \geq 0,\ \forall j {\in} [1{:} (J {+}2)], \ h_\ell (\boldsymbol{\xi}) \geq 0, \forall \, \indexl\},
        \end{align}where $\widetilde{g}_j^{(k)}$ are defined under Eqn. \eqref{eqn:sos-rep2}. Note that $\widetilde{\Xi}_k$ is non-empty as $(\boldsymbol{\xi}, \tau^{(k)}_1)\in \widetilde{\Xi}_k$ for any $\boldsymbol{\xi} \in \Xi$. Further, direct verification shows that  {\small \begin{eqnarray*}
            -\widetilde{\xi}^2 & = &  \frac{[\widetilde{\xi} - \tau^{(k)}_2]^2}{\tau^{(k)}_1 - \tau^{(k)}_2} \left(\tau^{(k)}_1 - \widetilde{\xi}\right)   + (\tau^{(k)}_1 + \tau^{(k)}_2) (\tau^{(k)}_1 - \widetilde{\xi}) + \frac{[ \tau^{(k)}_1 - \widetilde{\xi}\, ]^2}{\tau^{(k)}_1 - \tau^{(k)}_2} \left(\widetilde{\xi} - \tau^{(k)}_2\right) - (\tau^{(k)}_1)^2\\
            & = & \frac{[\widetilde{g}_{J+2}^{(k)} (\boldsymbol{\xi}, \widetilde{\xi}) ]^2}{\tau^{(k)}_1 - \tau^{(k)}_2} \widetilde{g}_{J+1}^{(k)} (\boldsymbol{\xi}, \widetilde{\xi})  + (\tau^{(k)}_1 + \tau^{(k)}_2) \widetilde{g}_{J+1}^{(k)} (\boldsymbol{\xi}, \widetilde{\xi}) + \frac{[ \widetilde{g}_{J+1}^{(k)} (\boldsymbol{\xi},\widetilde{\xi}) ]^2}{\tau^{(k)}_1 - \tau^{(k)}_2} \widetilde{g}_{J+2}^{(k)} (\boldsymbol{\xi},\widetilde{\xi}) - (\tau^{(k)}_1)^2, 
        \end{eqnarray*}}where the second equality follows by the definitions of $\widetilde{g}_{J+1}^{(k)}$ and $\widetilde{g}_{J+2}^{(k)}$.
        
        Since the Archimedean condition holds for $\{h_\ell\}_{\ell=1}^L$, there exist $R_0 \in \mathbb{N}$, $\sigma_\ell  \in \Sigma^2 (\boldsymbol{\xi})$, $\ell {\in} [0{:}L]$ such that $R_0-\|\boldsymbol{\xi}\|^2_2 = \sigma_0(\boldsymbol{\xi})+\sum_{\ell=1}^L \sigma_\ell (\boldsymbol{\xi})h_\ell (\boldsymbol{\xi})$. Choose $R_1 \in \mathbb{N}$ such that $R_1 \ge (\tau^{(k)}_1)^2$. This implies that 
\begin{align*}
& \ R_0+R_1 -\|(\boldsymbol{\xi},\widetilde{\xi})\|^2_2  
\\ = &\ R_0-\|\boldsymbol{\xi}\|^2_2 + R_1 -\widetilde{\xi}^2\\
 = &\  \sigma_0(\boldsymbol{\xi})+\sum_{\ell=1}^L \sigma_\ell (\boldsymbol{\xi})h_\ell (\boldsymbol{\xi}) + \frac{[\widetilde{g}_{J+2}^{(k)} (\boldsymbol{\xi}, \widetilde{\xi}) ]^2}{ \tau^{(k)}_1 - \tau_2^{(k)}} \widetilde{g}_{J+1}^{(k)} (\boldsymbol{\xi}, \widetilde{\xi}) + (\tau_1^{(k)} + \tau_2^{(k)}) \widetilde{g}_{J+1}^{(k)} (\boldsymbol{\xi}, \widetilde{\xi}) \\
& \qquad + \frac{[\widetilde{g}_{J+1}^{(k)} (\boldsymbol{\xi}, \widetilde{\xi}) ]^2 }{\tau^{(k)}_1 - \tau_2^{(k)}} \widetilde{g}_{J+2}^{(k)} (\boldsymbol{\xi}, \widetilde{\xi}) + (R_1 -(\tau^{(k)}_1)^2),
\end{align*}which means that the Archimedean condition also holds for $(\{h_\ell\}_{\ell=1}^L \cup \{\widetilde{g}_{J+1}^{(k)}, \widetilde{g}_{J+2}^{(k)} \}) \subseteq \R[(\boldsymbol{\xi},\widetilde{\xi})]$, and thus for $(\{h_\ell\}_{\ell=1}^L \cup \{\widetilde{g}_j^{(k)}\}_{j=1}^{J+2}) \subseteq \R[(\boldsymbol{\xi},\widetilde{\xi})]$.

    Consider the function $\varphi (\boldsymbol{\xi},\widetilde{\xi}):= \widetilde{\xi} + f ({\boldsymbol{\xi}})$  for $(\boldsymbol{\xi},\widetilde{\xi})\in\R^m\times \R$. Under the validation of Eqn. \eqref{3.1-e-them}, we can verify that $\varphi (\boldsymbol{\xi},\widetilde{\xi})>0$ for all $(\boldsymbol{\xi},\widetilde{\xi})\in \widetilde{\Xi}_k$. Therefore, by \Cref{prop:Putinar}, there exist $\sigma_\ell ^{(k)} \in \Sigma^2(\boldsymbol{\xi},\widetilde{\xi})$, $\ell {\in} [0{:}L]$, $\eta_{j}^{(k)}  \in \Sigma^2(\boldsymbol{\xi},\widetilde{
    \xi}),\ j {\in} [1{:} (J {+}2)]$ such that \begin{align*}
        \varphi = \sigma_0^{(k)} + \sum_{\ell=1}^L \sigma_\ell ^{(k)} h_\ell + \sum_{j=1}^{J+2} \eta_j^{(k)} \widetilde{g}_j^{(k)}, 
    \end{align*}or in other words, \begin{align*}
        \sigma_0^{(k)} = \varphi - \sum_{\ell=1}^L \sigma_\ell ^{(k)} h_\ell - \sum_{j=1}^{J+2} \eta_{j}^{(k)} \widetilde{g}_j^{(k)} \in \Sigma^2 (\boldsymbol{\xi}, \widetilde{\xi}), 
    \end{align*}arriving at the conclusion.   
\end{proof}

\begin{remark}
While the above positivity representation is established using the celebrated Putinar’s Positivstellensatz, it provides a flexible and convenient tool for deriving the asymptotic convergence of the SDP hierarchy used to compute the optimal value of the expectation problem with piecewise polynomial loss functions over Wasserstein balls.

In passing, we also note that by letting $K=J=1$, $g_1^{(1)} \equiv 0$ in \Cref{thm:positivity-representation} and setting $\widetilde{\xi}=0$, \begin{align*}
    \tau_1^{(1)} = 1 >\max_{\boldsymbol{\xi} \in \Xi} g_1^{(1)} (\boldsymbol{\xi})=0\quad \text{and} \quad \tau_2^{(1)} = -1 < \min_{\boldsymbol{\xi} \in \Xi} g_1^{(1)}(\boldsymbol{\xi})=0,
\end{align*}
the conclusion of \Cref{thm:positivity-representation} collapses to the one in the classical Putinar's Positivstellensatz (see \Cref{prop:Putinar}).
\end{remark}

The next subsection moves to presenting SDP hierarchies for the expectation problems.  

% ================================================
\subsection{Converging Hierarchy of SDPs} \label{section:hierarchy}

We first recall the expectation problem:
\begin{align} \tag{EP} \label{problem:mmt-minimax}
    \min_{\mathbb{Q} \in \mathbb{B}_\varepsilon(\widehat{\mathbb{P}} )} \mathbb{E}^{\mathbb{Q}} \Big [ \min_{\indexk} \max_{\indexj} g_j^{(k)}(\boldsymbol{\xi}) \Big ].
\end{align}

The support set $\Xi = \{\boldsymbol{\xi} \in \R^m : h_\ell (\boldsymbol{\xi}) \geq 0,\ \forall \indexl\}$ is a basic semi-algebraic set as in \Cref{def:supportset}. Recall that $\mathbb{B}_\varepsilon (\widehat{\mathbb{P}}) = \{\mathbb{Q} \in \mathcal{P} (\Xi) : d_W (\mathbb{Q}, \widehat{\mathbb{P}}) \leq \varepsilon \}$ is the Wasserstein ambiguity set with $\widehat{\mathbb{P}} = \frac{1}{N} \Sumi \1_{\widehat{\boldsymbol{\xi}}_i}$ the empirical distribution of past realizations $\widehat{\boldsymbol{\xi}}_i \in \Xi$, $\indexi$ obtained from datasets. 
The basic semi-algebraic support set bound by inequalities has also appeared in \cite{lasserre2008semidefinite,nie2023distributionally}. It covers some of the most practical support sets, such as the ellipsoid \cite{delage2010distributionally} and the box \cite{mohajerin2018data}. 

Throughout this subsection, we assume that the Archimedean condition (Assumption~\ref{assump:archimedean}) holds for $\{h_\ell\}_{\ell=1}^L$. This implies that $\Xi$ is compact, and so $\tau_1^{(k)}$ and $\tau_2^{(k)}$ as in \eqref{eqn:tauk-conditions} exist for each $\indexk$.

We associate \eqref{problem:mmt-minimax} with the following hierarchy of SDP programs, for a given $r\in\mathbb{N}$: 
\begin{align*}\tag{AD$_r$}\label{problem:mmtd-rminmax}
    \max_{\substack{ \lambda \geq 0, \alpha_i \in \R\\  \eta_{j,i}^{(k)} , \, \sigma_{\ell,i}^{(k)}\in \Sigma^2 }} &
     -\lambda \varepsilon^2 + \frac{1}{N} \Sumi \alpha_i \\ 
    \st \hspace{5mm} & \psi +\lambda \varphi_i - \alpha_i  - \sum_{\ell=1}^L \sigma_{\ell,i}^{(k)} h_\ell - \sum_{j=1}^{J+2} \eta_{j,i}^{(k)} \widetilde{g}_j^{(k)} \in \Sigma_{2r}^2(\boldsymbol{\xi},\widetilde{\xi}), \quad \forall \indexk, \, \indexi, \\ & \deg\big(\sigma_{\ell,i}^{(k)}h_\ell \big) \leq 2r,\quad \deg\big(\eta_{j,i}^{(k)}\widetilde{g}_{j}^{(k)}\big) \le 2r, \quad \forall \indexk,\, \indexl, \, j {\in} [1{:}(J{+}2)],\, \indexi,
\end{align*}where $\psi (\boldsymbol{\xi}, \widetilde{\xi}) = \widetilde{\xi}$, $\varphi_i (\boldsymbol{\xi}, \widetilde{\xi}) = \| \boldsymbol{\xi} - \widehat{\boldsymbol{\xi}}_i \|_2^2$, $\widetilde{g}_j^{(k)} (\boldsymbol{\xi}, \widetilde{\xi}) = \widetilde{\xi} - g_j^{(k)} (\boldsymbol{\xi})$ for each $\indexj$, $\widetilde{g}_{J+1}^{(k)} (\boldsymbol{\xi}, \widetilde{\xi}) = \tau^{(k)}_1 - \widetilde{\xi}$ and $\widetilde{g}_{J+2}^{(k)} (\boldsymbol{\xi}, \widetilde{\xi}) = \widetilde{\xi} - \tau_2^{(k)}$. Here, $\widehat{\boldsymbol{\xi}}_i \in \Xi$, $\indexi$, are the past realizations obtained from datasets as given in the definition of the Wasserstein ambiguity set.

We let $\max \eqref{problem:mmtd-rminmax} = -\infty$ if its feasibility region is empty. We know that the feasibility region enlarges as the step $r$ increases, and so the feasibility region will eventually become non-empty.  

While the constraints in \eqref{problem:mmtd-rminmax} are written as sum-of-squares constraints, they admit an equivalent representation as linear matrix inequalities \cite[Proposition~2.1]{lasserre2015introduction}, and so, \eqref{problem:mmtd-rminmax} can be written as a semi-definite program. Accordingly, we refer to our hierarchy as a hierarchy of SDP programs.

\begin{theorem}[\textbf{Asymptotic convergence}]\label{thm:SOSconvergence-minimax}

For the problem \eqref{problem:mmt-minimax} under the Wasserstein ambiguity set $\mathbb{B}_\varepsilon (\widehat{\mathbb{P}})$ and support set $\Xi$ in \eqref{def:supportset}, let $\Xi$ be a convex set; let the Archimedean condition (\Cref{assump:archimedean}) hold; and let $\varepsilon > 0$. Then, \[\max \text{\eqref{problem:mmtd-rminmax}} \le \max({\rm AD}_{r+1}) \le \min \eqref{problem:mmt-minimax} \] for all $r \in \mathbb{N}.$ Moreover, $\displaystyle \lim_{r \to \infty} \max\eqref{problem:mmtd-rminmax} = \min \eqref{problem:mmt-minimax}$.
\end{theorem}

\begin{proof}
Since the Archimedean condition (Assumption~\ref{assump:archimedean}) holds for $\{h_\ell\}_{\ell=1}^L$, $\Xi$ is compact. 
We first observe that, by construction, $\max \text{\eqref{problem:mmtd-rminmax}}  \le \max \, ({\rm AD}_{r+1})$ for each given $r \in \mathbb{N}$. 

Next, we show $\max \eqref{problem:mmtd-rminmax} \le \min \eqref{problem:mmt-minimax}$ for any given $r \in \mathbb{N}$. To see this, note from \Cref{lemma:semi-inf} with $\displaystyle g =\min_{\indexk} \max_{\indexj} g_j^{(k)}$ that $\min \eqref{problem:mmt-minimax}=\max \eqref{eqn:use1}$, where $\eqref{eqn:use1}$ is given by 
\begin{align}\label{eqn:use1}
    \max_{\substack{ \lambda \geq 0, \alpha_i \in \R}} & 
     -\lambda \varepsilon^2 + \frac{1}{N} \sum_{i=1}^N \alpha_i \\ {\rm s.t.} \ \ & \min_{\boldsymbol{\xi} \in \Xi} \Big \{ \lambda \| \boldsymbol{\xi} - \widehat{\boldsymbol{\xi}}_i \|_2^2  - \alpha_i + \min_{\indexk}\max_{\indexj}\, g_j^{(k)} (\boldsymbol{\xi} ) \Big \} \ge 0, \ \forall i {\in} [1{:}N] . \nonumber
\end{align}

Thus, the conclusion follows if one shows that $\max \eqref{problem:mmtd-rminmax} \leq \max \eqref{eqn:use1}$ for any given $r \in \mathbb{N}$. Now, take any feasible point $\lambda\geq 0$, $\alpha_i \in \R$, $\eta_{j,i}^{(k)}$, $\sigma_{\ell,i}^{(k)} \in \Sigma^2_{2r} (\boldsymbol{\xi}, \widetilde{\xi})$ from the problem \eqref{problem:mmtd-rminmax}, then, \begin{align*}
    \widetilde{\xi} + \lambda \|\boldsymbol{\xi} - \widehat{\boldsymbol{\xi}}_i\|_2^2 - \alpha_i - \sum_{\ell=1}^L \sigma_{\ell,i}^{(k)}(\boldsymbol{\xi}, \widetilde{\xi}) h_\ell(\boldsymbol{\xi}) - \sum_{j=1}^{J+2} \eta_{j,i}^{(k)}(\boldsymbol{\xi}, \widetilde{\xi}) \widetilde{g}_j^{(k)}(\boldsymbol{\xi}, \widetilde{\xi}) \geq 0,
\end{align*}for all $(\boldsymbol{\xi}, \widetilde{\xi}) \in \widetilde{\Xi}_k$, $\indexk$, $i {\in} [1{:}N]$, where $\widetilde{\Xi}_k$ is defined in Eqn. \eqref{tilde-O-pos}.

Fix an index $\indexk$, $i {\in} [1{:}N]$. The non-negativity over $\widetilde{\Xi}_k$ means that $\widetilde{\xi} + \lambda \|\boldsymbol{\xi} -  \widehat{\boldsymbol{\xi}}_i \|_2^2- \alpha_i \geq 0$ for all $(\boldsymbol{\xi}, \widetilde{\xi}) \in \widetilde{\Xi}_k$. Because $\displaystyle (\boldsymbol{\xi}, \max_{\indexj} g_j^{(k)} (\boldsymbol{\xi})) \in \widetilde{\Xi}_k$ for any $\boldsymbol{\xi} \in \Xi$, so \begin{align*}
    \max_{\indexj} g_j^{(k)} (\boldsymbol{\xi}) + \lambda \| \boldsymbol{\xi} - \widehat{\boldsymbol{\xi}}_i\|_2^2 - \alpha_i \geq 0
\end{align*}
for all $\boldsymbol{\xi} \in \Xi$. Since this is true for any $\indexk$, so $\displaystyle \lambda \|\boldsymbol{\xi} - \widehat{\boldsymbol{\xi}}_i \|_2^2 - \alpha_i + \min_{\indexk} \max_{\indexj} g_j^{(k)} (\boldsymbol{\xi}) \geq 0$ for all $\boldsymbol{\xi} \in \Xi$, $i {\in} [1{:}N]$. This implies that the chosen $\lambda\geq 0$, $\alpha_i \in \R$ is feasible for the problem \eqref{eqn:use1}. As the problems \eqref{eqn:use1} and \eqref{problem:mmtd-rminmax} share the same objective function, one has  $\max \eqref{problem:mmtd-rminmax} \leq \max \eqref{eqn:use1}$. Therefore, for any $r \in \mathbb{N}$, \begin{equation}\label{eqn:use2} \max \text{\eqref{problem:mmtd-rminmax}} \le \max({\rm AD}_{r+1}) \le \max\text{\eqref{eqn:use1}} = \min \eqref{problem:mmt-minimax}. 
\end{equation}

To prove the asymptotic convergence, we first notice from Eqn. \eqref{eqn:use2} that \begin{align*}
    \limsup_{r \rightarrow \infty}\max \text{\eqref{problem:mmtd-rminmax}} \le \max\text{\eqref{eqn:use1}} = \min  \eqref{problem:mmt-minimax}.
\end{align*}

So, the asymptotic convergence will follow if one shows that $\displaystyle \liminf_{r \rightarrow \infty}\max \text{\eqref{problem:mmtd-rminmax}} \ge \max\text{\eqref{eqn:use1}}$. To see this, take an arbitrary $\gamma>0$. From the attainment of the problem \eqref{eqn:use1},
there exist $\lambda^\star \geq 0$ and $\alpha_i^\star\in \R$, $i {\in} [1{:}N]$, such that \begin{align*}
    \max(\text{\ref{eqn:use1}})-\frac{\gamma}{2}  \le -\lambda^\star \varepsilon^2 + \frac{1}{N} \sum_{i=1}^N \alpha_i^\star
\end{align*}
and  \begin{align*}
    \min_{\boldsymbol{\xi} \in \Xi} \Big \{ \lambda^\star \| \boldsymbol{\xi} - \widehat{\boldsymbol{\xi}}_i \|_2^2 - \alpha_i^\star + \min_{\indexk} \max_{\indexj} \, g_j^{(k)}(\boldsymbol{\xi} )  \Big \} \ge 0, \, \ \forall i {\in} [1{:}N].
\end{align*}
This implies that $\max(\text{\ref{eqn:use1}})-\gamma \le  -\lambda^\star \varepsilon^2 + \frac{1}{N} \sum_{i=1}^N (\alpha_i^\star- \frac{\gamma}{2})$ and 
\begin{align} \label{eqn:sos-existence2}
    \min_{\boldsymbol{\xi} \in \Xi} \Big \{ \lambda^\star \| \boldsymbol{\xi} - \widehat{\boldsymbol{\xi}}_i \|_2^2  - \big(\alpha_i^\star-\frac{\gamma}{2}\big) + \min_{ \indexk} \max_{\indexj} \, g_j^{(k)} (\boldsymbol{\xi} )  \Big \} > 0, \ \forall i {\in} [1{:}N].
\end{align}
For each $i {\in} [1{:}N]$ and $\indexk$, by \Cref{thm:positivity-representation}, Eqn. \eqref{eqn:sos-existence2} implies that there exist \begin{align*}
    \eta_{j,i}^{(k)} \in \Sigma^2 (\boldsymbol{\xi},\widetilde{\xi}), \ j {\in} [1{:}(J{+}2)], \ \sigma_{\ell,i}^{(k)} \in \Sigma^2 (\boldsymbol{\xi},\widetilde{\xi}), \ \indexl, 
\end{align*}
such that  \begin{align*}
     \psi+ \lambda^\star \varphi_i  - \big(\alpha_i^\star- \frac{\gamma}{2}\big) - \sum_{\ell=1}^L \sigma_{\ell,i}^{(k)} h_\ell  - \sum_{j=1}^{J+2} \eta_{j,i}^{(k)} \widetilde{g}_j^{(k)}  \in \Sigma_{2 r_0}^2  (\boldsymbol{\xi}, \widetilde{\xi}),
\end{align*}
for some $r_0 \in \mathbb{N}$ 
with $\deg{(\sigma_{\ell,i}^{(k)}  h_\ell)}\le 2r_0, \, \indexl$ and $\deg{(\eta_{j,i}^{(k)}  \widetilde{g}_j^{(k)})} \le 2r_0, \,j {\in} [1{:} (J{+}2)]$. Thus, for each $\gamma>0$, there exists $r_0 \in \mathbb{N}$ such that $(\lambda^\star, \alpha_i^\star-\frac{\gamma}{2}, \eta_{j,i}^{(k)}, \sigma_{\ell,i}^{(k)})$ is feasible for (${\rm AD}_{r_0}$), which implies that $\max(\text{\ref{eqn:use1}})-\gamma \le \max({\rm AD}_{r_0})$. So,  $\displaystyle \liminf_{r \rightarrow \infty}\max \text{\eqref{problem:mmtd-rminmax}} \ge \max\text{\eqref{eqn:use1}}$, and hence  $\displaystyle \lim_{r \to \infty} \max\text{\eqref{problem:mmtd-rminmax}} = \max\text{\eqref{eqn:use1}} = \min \eqref{problem:mmt-minimax}$.
\end{proof}

When $J=1$, the objective function takes the form $g(\boldsymbol{\xi}) =  \min_{\indexk} g_1^{(k)}(\boldsymbol{\xi})$, where $g_1^{(k)}$'s are polynomials. Using \Cref{thm:positivity-representation}, the SOS constraints are simplified, and the hierarchy of SDP programs in \eqref{problem:mmtd-rminmax} becomes:\begin{align*}\tag{$\widetilde{\rm AD}_r$}\label{problem:mmtd-rmin}
    \max_{\substack{ \lambda \geq 0, \alpha_i \in \R, \sigma_{\ell,i}^{(k)}\in\Sigma^2 }} &
     -\lambda \varepsilon^2 + \frac{1}{N} \Sumi \alpha_i \\ \st \hspace{5mm} & g_1^{(k)} + \lambda \varphi_i  - \alpha_i - \sum_{\ell=1}^L \sigma_{\ell,i}^{(k)} h_\ell\in \Sigma_{2r}^2(\boldsymbol{\xi}), \ \forall \indexk, \ \indexi,\\ & \deg\big(\sigma_{\ell,i}^{(k)}h_\ell \big) \leq 2r, \quad \forall \indexk,\, \indexl, \, \indexi, 
\end{align*}
for a given $r \in\mathbb{N}$ with $\varphi_i (\boldsymbol{\xi}) = \| \boldsymbol{\xi} - \widehat{\boldsymbol{\xi}}_i \|_2^2$. 

\begin{corollary}[\bf Asymptotic convergence for min-polynomials]\label{thm:SOSconvergence-min}
For the problem \eqref{problem:mmt-minimax}, let $\Xi = \{\boldsymbol{\xi} \in \R^m : h_\ell(\boldsymbol{\xi}) \geq 0,\, \forall \indexl \}$ as in \eqref{def:supportset}; let \Cref{assump:archimedean} hold; let $J=1$ and $g_1^{(k)}: \R^m \to \R$ be a polynomial for each $\indexk$; and let $\varepsilon > 0$. Then, \[\max \text{\eqref{problem:mmtd-rmin}} \le \max(\widetilde{\rm AD}_{r+1}) \le  \min \eqref{problem:mmt-minimax} \] for all $r \in \mathbb{N}.$ Moreover, $\displaystyle \lim_{r\to\infty} \max\eqref{problem:mmtd-rmin} = \min \eqref{problem:mmt-minimax}$.
\end{corollary}

\begin{proof}
As in the proof of \Cref{thm:SOSconvergence-minimax}, $\max \eqref{problem:mmtd-rmin} \leq \max \eqref{problem:mmtd}$ for any given degree $r \in \mathbb{N}$. This follows from the fact that, for each $\indexi$, \begin{align*}
    g_1^{(k)} + \lambda \varphi_i - \alpha_i - \sum_{\ell=1}^L \sigma_{\ell,i}^{(k)}  h_\ell \in \Sigma^2_{2r} (\boldsymbol{\xi}),\ \forall \indexk,
\end{align*} which, in turn, implies that \begin{align*}
    \lambda \| \boldsymbol{\xi} - \widehat{\boldsymbol{\xi}}_i \|_2^2 - \alpha_i + \min_{\indexk} g_1^{(k)} (\boldsymbol{\xi}) \geq 0, \ \forall \boldsymbol{\xi} \in \Xi,
\end{align*}for any $\lambda \geq 0$, $\alpha_i \in \R$. 
    
    Hence, $\max \text{\eqref{problem:mmtd-rmin}} \le \max(\widetilde{\rm AD}_{r+1}) \le \max\text{\eqref{problem:mmtd}} = \min \eqref{problem:mmt-minimax}$ holds, and the convergence result follows the same line of arguments as in the proof of asymptotic convergence of \Cref{thm:SOSconvergence-minimax}.
\end{proof}

%%%%%%%%%%%%%%%%%%%%%%%%%%%%%%%%%%%%%
\section{Piecewise Convex Polynomials and Finite Convergence}\label{sec:finite-cvg}

Recall that we are considering the setting where the loss function is given by
\[
g(\boldsymbol{\xi}) = \min_{\indexk} \max_{\indexj} g_j^{(k)}(\boldsymbol{\xi}),
\]
and the support set takes the form $\Xi = \{ \boldsymbol{\xi} \in \mathbb{R}^m : h_\ell(\boldsymbol{\xi}) \geq 0,\, \forall \indexl\}$ as defined in Eqn. \eqref{def:supportset}. 

In this section, we assume that $g_j^{(k)}\!, \indexj, \indexk$, and $(-h_\ell), \indexl$, are convex polynomials, and show that the hierarchy for the problem \eqref{problem:mmt-minimax} exhibits finite convergence. It is worth noting that a piecewise convex polynomial is not necessarily a convex function.

We associate with \eqref{problem:mmt-minimax} the following hierarchy of SDP programs:
\begin{align*}\tag{FD$_r$}\label{problem:mmtd-r-convex}
    \max_{\substack{ \lambda \geq 0, \alpha_i \in \R\\ \boldsymbol{\delta}_i^{(k)} \in 
    \Delta, \sigma_{\ell,i}^{(k)} \in \Sigma^2_{2r} (\boldsymbol{\xi})}} &
     -\lambda \varepsilon^2 + \frac{1}{N} \sum_{i=1}^N \alpha_i \\ 
    {\rm s.t.} \hspace{8mm} & \ \Sumj \delta_{j,i}^{(k)} g_j^{(k)}+ \lambda \varphi_i  - \alpha_i - \sum_{\ell=1}^L \sigma_{\ell,i}^{(k)} h_\ell \in \Sigma_{2r}^2 (\boldsymbol{\xi}), \quad \forall \indexk, \ i {\in} [1{:}N], \\ & \deg\big(\sigma_{\ell,i}^{(k)}h_\ell \big) \leq 2r, \quad \forall \indexk, \, \indexl, \, i {\in} [1{:}N], 
\end{align*}where $\varphi_i (\boldsymbol{\xi}) = \| \boldsymbol{\xi} - \widehat{\boldsymbol{\xi}}_i \|_2^2$ and $r \in \mathbb{N}$ is a given number with $\deg\big(g_j^{(k)}\big) \le 2r$, $\forall \indexj$, $\indexk$. We let $\max \eqref{problem:mmtd-r-convex} = -\infty$ if its feasibility region is empty. 

We note that, in comparison to the SDP hierarchy \eqref{problem:mmtd-rminmax} for general polynomials, the problem \eqref{problem:mmtd-r-convex} contains fewer SOS polynomial variables. The SOS polynomials in the problem \eqref{problem:mmtd-rminmax} are $(m +1)$-variate compared to $m$-variate in \eqref{problem:mmtd-r-convex}. Moreover, the multipliers $\eta_{j,i}^{(k)} \in \Sigma^2 (\boldsymbol{\xi}, \widetilde{\xi})$, $\indexj$, for the problem \eqref{problem:mmtd-rminmax} are replaced with $\boldsymbol{\delta}_i^{(k)} \in \Delta$, which makes \eqref{problem:mmtd-r-convex} easier to solve. 

We show that the hierarchy in \eqref{problem:mmtd-r-convex} exhibits finite convergence.

\begin{theorem}[\textbf{Finite convergence}]\label{thm:wasserstein-finite}

   For the problem \eqref{problem:mmt-minimax} under the Wasserstein ambiguity set $\mathbb{B}_\varepsilon (\widehat{\mathbb{P}})$ and support set $\Xi$ in \eqref{def:supportset}, let $(-h_\ell)$, $\indexl$, $g_j^{(k)}$, $\indexj$, $\indexk$, be convex polynomials; let the Archimedean condition (\Cref{assump:archimedean}) hold; and let $\varepsilon > 0$. Suppose further that: \begin{itemize}
       
       \item (Slater's condition) There exists $\overline{\boldsymbol{\xi}} \in \R^m$ such that $h_\ell(\overline{\boldsymbol{\xi}}) > 0$ for all $\indexl$; and 

       \item (Hessian condition) For each $\indexk$, the Hessian of $\Sumj a_j^{(k)}g_j^{(k)}$ is positive definite at every minimizer of $\Sumj a_j^{(k)}g_j^{(k)}$ in $\Xi$, for all $\boldsymbol{a}^{(k)} = (a_1^{(k)},\ldots,a_L^{(k)}) \in \Delta$. 
   \end{itemize}
   
   Then, for all $r \in \mathbb{N}$, 
    \begin{equation*}
    \max \text{\eqref{problem:mmtd-r-convex}} \le \max({\rm FD}_{r+1}) \le \min \eqref{problem:mmt-minimax}.
    \end{equation*}
    Moreover, there exists $r^\star \in \mathbb{N}$ such that $\max({\rm FD}_{r^\star}) = \min \eqref{problem:mmt-minimax}$.
\end{theorem}

Before proceeding with the proof, we note that the Hessian condition is automatically satisfied when $g_j^{(k)}$ is strongly convex for each $\indexk, \indexj$.

\begin{proof} 

Since the Archimedean condition (Assumption~\ref{assump:archimedean}) holds for $\{h_\ell\}_{\ell=1}^L$, $\Xi$ is compact. By construction, $\max (\text{\ref{problem:mmtd-r-convex}}) \le \max({\rm FD}_{r+1})$ for any $r \in \mathbb{N}$. 

We start by proving $\max (\text{\ref{problem:mmtd-r-convex}}) \le \max \eqref{eqn:use1}$ for any $r \in \mathbb{N}$. Fix an index $\indexk$, $i {\in} [1{:}N]$. Take any feasible point $\lambda \geq 0$, $\alpha_i \in \R$, $\boldsymbol{\delta}_i^{(k)} \in \Delta$, $\sigma_{\ell,i}^{(k)} \in \Sigma^2_{2r} (\boldsymbol{\xi})$ from the problem \eqref{problem:mmtd-r-convex}. The SOS constraint in \eqref{problem:mmtd-r-convex} gives us that \begin{align*}
   \Sumj \delta_{j,i}^{(k)} g_j^{(k)} (\boldsymbol{\xi}) + \lambda \| \boldsymbol{\xi} - \widehat{\boldsymbol{\xi}}_i \|_2^2 - \alpha_i \geq 0, 
\end{align*}for all $\boldsymbol{\xi} \in \Xi$. Since $\displaystyle \Sumj \delta_{j,i}^{(k)} g_j^{(k)} (\boldsymbol{\xi}) \leq \max_{\indexj} g_j^{(k)} (\boldsymbol{\xi})$ for any $\boldsymbol{\xi} \in \R^m$, it then follows that $\displaystyle \lambda \| \boldsymbol{\xi} - \widehat{\boldsymbol{\xi}}_i\|_2^2 - \alpha_i + \max_{\indexj} g_j^{(k)} (\boldsymbol{\xi}) \geq 0$ for all $\boldsymbol{\xi} \in \Xi$. This holds for any $\indexk$, so \begin{align} \label{eqn:nonnegativity-constraint}
    \lambda \| \boldsymbol{\xi} - \widehat{\boldsymbol{\xi}}_i \|_2^2 - \alpha_i + \min_{\indexk} \max_{\indexj} g_j^{(k)} (\boldsymbol{\xi}) \geq 0, \quad \text{for all} \quad \boldsymbol{\xi} \in \Xi. 
\end{align}

This implies that $\lambda \geq 0$, $\alpha_i\in \R$ is feasible for the problem \eqref{eqn:use1} for $i {\in} [1{:}N]$, and we obtain $\max \eqref{problem:mmtd-r-convex} \leq \max \eqref{eqn:use1}$. Consequently, we have \begin{align*}
    \max (\text{\ref{problem:mmtd-r-convex}}) \le \max({\rm FD}_{r+1}) \le \max(\text{\ref{eqn:use1}}) = \min \eqref{problem:mmt-minimax},
\end{align*} 
where $\min \eqref{problem:mmt-minimax} = \max \eqref{eqn:use1}$ follows from \Cref{lemma:semi-inf}.     

To prove finite convergence, let $(\lambda^\star, \alpha_i^\star)$ be the optimal solution for the semi-infinite problem \eqref{eqn:use1} with attainment of $\max \eqref{eqn:use1}$ followed from \Cref{lemma:semi-inf} and let $i {\in} [1{:}N]$ be fixed. The non-negativity constraint in \Cref{eqn:nonnegativity-constraint} is equivalent to \begin{align}\label{eqn:max-to-sum}
    & \min_{\boldsymbol{\xi} \in \Xi} \left\{\max_{\indexj}g_j^{(k)} (\boldsymbol{\xi}) + \lambda^\star \| \boldsymbol{\xi} - \widehat{\boldsymbol{\xi}}_i \|_2^2 - \alpha_i^\star \right\} \nonumber \\ = &\, \min_{\boldsymbol{\xi} \in \Xi}\max_{\boldsymbol{\delta} \in \Delta} \underbrace{\left\{\Sumj \delta_j g_j^{(k)} (\boldsymbol{\xi}) + \lambda^\star \| \boldsymbol{\xi} - \widehat{\boldsymbol{\xi}}_i \|_2^2 - \alpha_i^\star \right\}}_{P_i^{(k)} (\boldsymbol{\xi},\boldsymbol{\delta})} \geq 0 
\end{align}for each $\indexk$. Notice that $P_i^{(k)} (\cdot,\boldsymbol{\delta})$ is convex for each fixed $\boldsymbol{\delta} \in \Delta \subset  \R^J$, and $P_i^{(k)} (\boldsymbol{\xi},\cdot)$ is affine for each fixed $\boldsymbol{\xi} \in \Xi \subset \R^m$. Moreover,  both the simplex $\Delta$ and $\Xi$ are convex compact sets. Invoking the Convex-Concave Minimax Theorem \cite[cf.][Theorem~2.10.2]{zalinescu2002convex} yields  
$\displaystyle \max_{\boldsymbol{\delta} \in \Delta} \min_{\boldsymbol{\xi} \in \Xi} P_i^{(k)} (\boldsymbol{\xi},\boldsymbol{\delta})   \geq 0$, and so there exist $\boldsymbol{\delta}_i^{(k)} \in \Delta$, $\indexk$, such that \begin{align} \label{eqn:nonnegativity-finite}
& \min_{\boldsymbol{\xi} \in \Xi} P_i^{(k)} (\boldsymbol{\xi},\boldsymbol{\delta}_i^{(k)}) \geq 0 \nonumber \\ \iff \ &  \ \widetilde{P}_{i}^{(k)}(\boldsymbol{\xi}) := \Sumj \delta_{j,i}^{(k)} g_j^{(k)} (\boldsymbol{\xi}) + \lambda^\star \| \boldsymbol{\xi} - \widehat{\boldsymbol{\xi}}_i \|_2^2 - \alpha_i^\star  \geq 0 \text{ for all } \boldsymbol{\xi} \in \Xi, \ \indexk.     
\end{align} 

Let $\indexk$ be fixed and let $v^{i,k}_{\text{min}} := \min_{\boldsymbol{\xi} \in \Xi}  \widetilde{P}_i^{(k)} (\boldsymbol{\xi})$, which is non-negative by Eqn. \eqref{eqn:nonnegativity-finite}. By Lagrangian duality of convex optimization \cite[cf.][Theorem~2.9.2]{zalinescu2002convex} under Slater's condition, there exist $\theta_{\ell,i}^{(k)}  \in \R_+$, $\indexl$, such that
\begin{align}\label{eqn:pminineq}
& v^{i,k}_{\text{min}} = \min_{\boldsymbol{\xi} \in \R^m} \left\{ \widetilde{P}_i^{(k)}(\boldsymbol{\xi})  - \Suml \theta_{\ell,i}^{(k)} h_\ell(\boldsymbol{\xi})\right\} \nonumber \\ \Longleftrightarrow\ &\  v_{\rm min}^{i,k} \leq \widetilde{P}_i^{(k)} (\boldsymbol{\xi}) - \Suml \theta_{\ell,i}^{(k)} h_\ell (\boldsymbol{\xi})\ \ \text{for all}\ \ \boldsymbol{\xi} \in \R^m.
\end{align}
This gives us that the following convex polynomial (since $\widetilde{P}_{i}^{(k)}$ and $(-h_\ell)$, $\indexl$, are convex, and $\theta_{\ell,i}^{(k)} \ge 0$, $\indexl$):
\[\widetilde{f}_{i}^{(k)}(\boldsymbol{\xi}) := \widetilde{P}_i^{(k)}(\boldsymbol{\xi})  - \Suml \theta_{\ell,i}^{(k)} h_\ell(\boldsymbol{\xi}) - v^{i,k}_{\text{min}}\]
is non-negative for all $\boldsymbol{\xi} \in \R^m$. 

If $\boldsymbol{\xi}^\star$ is a minimizer of $\widetilde{P}_i^{(k)}$ in $\Xi$, then
$v^{i,k}_{\text{min}} = \widetilde{P}_i^{(k)}(\boldsymbol{\xi}^\star)$. It follows from Eqn. \eqref{eqn:pminineq} that $\Suml \theta_{\ell,i}^{(k)} h_\ell (\boldsymbol{\xi}^\star)=0$, and so,  $\boldsymbol{\xi}^\star$ is the minimizer of $\widetilde{f}_{i}^{(k)}$ on $\R^m$ with $\widetilde{f}_{i}^{(k)}(\boldsymbol{\xi}^\star) = 0$. 

Meanwhile, by the Archimedean condition of $\{h_\ell\}_{\ell=1}^L$, there exists $R \in \mathbb{N}$ such that
\[
R - \|\boldsymbol{\xi}\|_2^2 = \widehat{\sigma}_0(\boldsymbol{\xi}) + \Suml \widehat{\sigma}_\ell (\boldsymbol{\xi}) h_\ell(\boldsymbol{\xi}), \quad \widehat{\sigma}_\ell  \in \Sigma^2(\boldsymbol{\xi}),\ \  \ell {\in} [0{:}L].
\]
The set $S := \{\boldsymbol{\xi} : R - \|\boldsymbol{\xi}\|_2^2 \geq 0\}$ contains $\Xi$. Without loss of generality, we assume that $\Xi \subset {\rm interior} (S)$ (otherwise, replace $R$ in the definition of $S$ with $\widehat{R} \ge R$, and still have $\widehat{R} - \|\boldsymbol{\xi}\|_2^2 \in \mathcal{Q}(h_1,\ldots,h_L)$). 

We now verify that $\widetilde{f}_{i}^{(k)}$ satisfies the conditions \textit{(i)} and \textit{(ii)} of \Cref{prop:scheiderer} on the set $S$, and we consider the cases [$\lambda^\star = 0$] and [$\lambda^\star > 0$] separately.

[$\lambda^\star = 0$]. In this case, note that $\widetilde{P}_i^{(k)} = \Sumj \delta_{j,i}^{(k)} g_j^{(k)} - \alpha_i^\star$ is a convex function, and by the assumption that the Hessian of $\Sumj \delta_{j,i}^{(k)} g_j^{(k)}$ is positive definite at its minimizer, the global minimizer of $\widetilde{P}_i^{(k)}$ on $\Xi$ is unique. Moreover, if $\boldsymbol{\xi}^\star \in \Xi$ is the unique global minimer of $\widetilde{P}_i^{(k)}$, then $\boldsymbol{\xi}^\star$ is the unique zero of $\widetilde{f}_i^{(k)}$ on $\R^m$. Thus, $\boldsymbol{\xi}^\star$ is the only zero of $\widetilde{f}_i^{(k)}$ on $\R^m$, in particular it is the only zero in $\mbox{interior}(S)$. This verifies condition \textit{(i)} of \Cref{prop:scheiderer}.

Additionally, we note that the Hessian \begin{align*}
    \nabla^2 \widetilde{f}_i^{(k)} = \nabla^2 \Big [\Sumj \delta_{j,i}^{(k)} g_j^{(k)} - \alpha_i^\star - \Suml \theta_{\ell,i}^{(k)} h_\ell - v_{\min}^{i,k} \Big ]
\end{align*}
is also positive definite at the point $\boldsymbol{\xi}^\star$, which is the unique zero of $\widetilde{f}_i^{(k)}$. This verifies condition \textit{(ii)} of  \Cref{prop:scheiderer}.

[$\lambda^\star > 0$]. The Hessian \begin{align*}
    \nabla^2 \widetilde{f}_i^{(k)}  = \nabla^2 \Big [\Sumj \delta_{j,i}^{(k)} g_j^{(k)}  + \lambda^\star \varphi_i -\alpha_i^\star - \Suml \theta_{\ell,i}^{(k)} h_\ell  - v_{\min}^{i,k} \Big ], \quad \varphi_i (\boldsymbol{\xi}) = \| \boldsymbol{\xi} - \widehat{\boldsymbol{\xi}}_i\|_2^2,
\end{align*}
is positive definite at all $\boldsymbol{\xi} \in \R^m$ due to the strong convexity of $\lambda^\star \varphi_i$. Moreover, similar to the [$\lambda^\star = 0$] case, $\boldsymbol{\xi}^\star$ is the unique zero of $\widetilde{f}_i^{(k)}$ in ${\rm interior} (S)$. 

Since $\widetilde{f}_i^{(k)}$ is non-negative on $\R^m$, hence on $S$, then $\widetilde{f}_{i}^{(k)}$ satisfies the conditions of \Cref{prop:scheiderer} on the set $S$. Applying \Cref{prop:scheiderer} yields \begin{align*}
    \widetilde{f}_i^{(k)}(\boldsymbol{\xi}) = \widetilde{\sigma}_{0,i}^{(k)}(\boldsymbol{\xi}) + \widetilde{\sigma}_{1,i}^{(k)}(\boldsymbol{\xi}) (R - \|\boldsymbol{\xi}\|_2^2)
\end{align*}
for some $\widetilde{\sigma}_{0,i}^{(k)}, \widetilde{\sigma}_{1,i}^{(k)} \in \Sigma^2(\boldsymbol{\xi})$. 
    
    Substituting the Archimedean representation of $(R - \|\boldsymbol{\xi}\|_2^2)$  and then rearranging for $\widetilde{P}_i^{(k)}$, we obtain \begin{align*}
        \widetilde{P}_i^{(k)} (\boldsymbol{\xi})
        &=  \underbrace{v^{i,k}_{\text{min}} + \widetilde{\sigma}_{0,i}^{(k)}(\boldsymbol{\xi}) + \widetilde{\sigma}_{1,i}^{(k)}(\boldsymbol{\xi})\widehat{\sigma}_{0}(\boldsymbol{\xi})}_{\sigma_{0,i}^{(k)}(\boldsymbol{\xi})} + \Suml h_\ell(\boldsymbol{\xi}) \underbrace{\left[ \theta_{\ell,i}^{(k)} + \widetilde{\sigma}_{1,i}^{(k)}(\boldsymbol{\xi}) \widehat{\sigma}_\ell (\boldsymbol{\xi})\right]}_{\sigma_{\ell,i}^{(k)}(\boldsymbol{\xi})},
    \end{align*}where $\sigma_{\ell,i}^{(k)} \in \Sigma^2 (\boldsymbol{\xi})$, $\ell {\in} [0{:}L]$, since $v^{i,k}_{\text{min}} \geq 0$. Therefore, there exists $r^\star \in \mathbb{N}$ such that $\deg (\sigma_{0,i}^{(k)}) \leq 2 r^\star$ and the equation above leads to \begin{align*}
        \Sigma_{2r^\star}^2 (\boldsymbol{\xi}) \ni \sigma_{0,i}^{(k)} (\boldsymbol{\xi}) &= \widetilde{P}_i^{(k)} (\boldsymbol{\xi}) - \Suml \sigma_{\ell,i}^{(k)} (\boldsymbol{\xi}) h_\ell (\boldsymbol{\xi}) \\
        &= \Sumj \delta_{j,i}^{(k)} g_j^{(k)} (\boldsymbol{\xi}) + \lambda^\star \varphi_i (\boldsymbol{\xi}) - \alpha_i^\star - \Suml \sigma_{\ell,i}^{(k)} (\boldsymbol{\xi}) h_\ell (\boldsymbol{\xi}),
    \end{align*}or equivalently,     
\[
\Sumj \delta_{j,i}^{(k)} g_j^{(k)}  + \lambda^\star \varphi_i  - \alpha_i^\star - \Suml \sigma_{\ell,i}^{(k)} h_\ell \in \Sigma_{2 r^\star}^2(\boldsymbol{\xi}). 
\]

So $(\lambda^\star, \alpha_i^\star, \boldsymbol{\delta}_{i}^{(k)}, \sigma_{\ell,i}^{(k)})$ is feasible for $(\mathrm{FD}_{r^\star})$, and \begin{align*}
    \max({\rm FD}_{r^\star}) \geq -\lambda^\star \varepsilon^2 + \frac{1}{N} \sum_{i=1}^N \alpha_i^\star = \max\eqref{eqn:use1}. 
\end{align*}

    Since $\max\eqref{problem:mmtd-r-convex} \leq \max\eqref{eqn:use1}$ for all $r$, the equality holds at $r^\star$.
\end{proof}

\begin{remark}[Finite and one-step convergence]

In the special case where $J = K = 1$, our Hessian condition of \Cref{thm:wasserstein-finite} collapses to the condition that $g_1^{(1)}$ is positive definite at its minimizers in $\Xi$. Related conditions have been used to study finite convergence of the Lasserre hierarchy for convex polynomial problems (see \cite[Corollary~3.3]{de2011lasserre} or \cite[Theorem~3.3]{jeyakumar2014convergence}).            

        By assuming that $g_j^{(k)}$'s in the objective function and $(- h_\ell)$'s in the support set to be SOS-convex polynomials, the optimal value and solution of \eqref{problem:mmt-minimax} can be computed by a single SDP. 
\end{remark}

%%%%%%%%%%%%%%%%%%%%%%%%%%%%%%%%%%%%%
\section{{Numerical Experiments}} \label{sec:dro}

This section is dedicated to numerical experiments on concrete decision-making problems of revenue estimation in \Cref{sec:rev_max} and portfolio optimization in \Cref{sec:portfolio} to examine the computational tractability of the resulting SDP hierarchies and test their scalability across a range of problem sizes. In addition, \Cref{sec:strength-limit} discusses the strengths and limitations of the computational studies.

All numerical experiments are carried out in MATLAB, using YALMIP as the modeling framework and MOSEK as the optimization solver. Computations were performed on a mid-2024 Apple MacBook Air equipped with an Apple M3 chip (8-core CPU, 10-core GPU) and 16 GB of unified memory. In our implementation, the SOS constraints appearing in our formulations are provided directly to YALMIP using its native SOS interface. YALMIP automatically converts these SOS constraints into equivalent semi-definite constraints, and the resulting SDPs are solved using MOSEK.

% ============================
\subsection{Data-Driven Revenue Estimation} \label{sec:rev_max}

This experiment exemplifies the hierarchy in \eqref{problem:mmtd-rminmax} of \Cref{section:hierarchy}
to piecewise polynomials involving third-degree polynomials in the context of revenue estimation.

In revenue estimation, a merchant offers a random quantity $\xi \in [0, R] \subseteq \mathbb{R}$, where $R > 0$, of goods to $K$ customers, each of whom quotes a price depending on the quantity. The merchant may sell to only one customer at a time and aims to maximize revenue by accepting the highest bid.

We model the \textit{offer price} $f_k$ from the $k$\textsuperscript{th} customer ($\indexk$) using a third-order utility function:
\begin{align*}
    f_k(\xi) = 
    \begin{cases}
        a_k (\xi - b_k)^3 + d_k, & \text{if } 0 \leq \xi \leq b_k, \\
        d_k, & \text{if } \xi > b_k,
    \end{cases}
\end{align*}
where $a_k > 0$, $d_k, b_k\geq 0$, and $-a_k b_k^3 + d_k \ge 0$.

Cubic utility functions have been used in mathematical models in economics and finance \cite{benishay1987fourth,gerber1998utility}. They are concave and non-decreasing, and they incorporate information about the mean, variance, and skewness of the random quantity \cite{benishay1987fourth}. A related revenue estimation model with a quadratic utility function is studied in \cite{han2015convex}.

The merchant's \textit{revenue} is thus the maximum offer price given by $\max_{\indexk} f_k(\xi)$. Estimating an upper bound for the expected revenue with Wasserstein ambiguity can be regarded as an  expectation model: \begin{align} \label{problem:revenue_uq1}\tag{RP}
    \max_{\mathbb{Q} \in \mathbb{B}_\varepsilon (\widehat{\mathbb{P}})} \mathbb{E}^{\mathbb{Q}} \Big [ \max_{\indexk} f_k (\xi) \Big ],
\end{align}which can be written as \begin{align*}
    - \min_{\mathbb{Q} \in \mathbb{B}_\varepsilon (\widehat{\mathbb{P}})} \mathbb{E}^{\mathbb{Q}} \Big [ g(\xi) := \min_{\indexk} \max_{j  \in [1{:}2]} g_j^{(k)} (\xi) \Big ],
\end{align*}
where 
\begin{equation}\label{eqn:revmax-g}
g_1^{(k)}(\xi) = - a_k (\xi - b_k)^3 - d_k \ \text{ and } \ g_2^{(k)}(\xi) = -d_k, \text{ for } \indexk.
\end{equation}

We assume that the random quantity is governed by an unknown true distribution $\mathbb{P}$, approximated by the empirical distribution $\widehat{\mathbb{P}} = \frac{1}{N} \sum_{i=1}^N \1_{\widehat{\xi}_i}$ through a training dataset $\widehat{\Xi}_N = \{\widehat{\xi}_i\}_{i=1}^N \subset \mathbb{R}$ consisting of $N$ past observations of the random quantity $\xi$. Here, the support set $\Xi = [0,R] = \{\xi \in \R : h_1 (\xi) :=  \xi \geq 0, h_2 (\xi) := R- \xi \geq 0 \}$ corresponds to the range of the random supply quantity.
    
We now link \eqref{problem:revenue_uq1} to a hierarchy of SDP programs. Let $\tau^{(k)}_1$, $\tau^{(k)}_2 \in \R$ such that \begin{align*}
     \max_{\boldsymbol{\xi} \in \Xi} \max_{j{\in}[1{:2}]}  g_j^{(k)} (\boldsymbol{\xi}) < \tau^{(k)}_1 \quad \text{and} \quad \tau^{(k)}_2 < \min_{\boldsymbol{\xi} \in \Xi} \min_{j {\in}[1{:2}]} g_j^{(k)} (\boldsymbol{\xi}). 
\end{align*} Consider the hierarchy of SDP programs given by, for a given $r \in \mathbb{N}$:
{\small \begin{align*}\tag{RD$_r$}\label{problem:revmax-hierarchy}
    \min_{\substack{ \lambda \geq 0, \alpha_i \in \R\\  \eta_{j,i}^{(k)} , \, \sigma_{\ell,i}^{(k)} \in \Sigma^2 } } &
     \lambda \varepsilon^2 - \frac{1}{N} \Sumi \alpha_i \\ 
    \st \hspace{6mm} 
    &\psi + \lambda \varphi_i - \alpha_i  - \sigma_{1, i}^{(k)}  h_1  - \sigma_{2, i}^{(k)} h_2 - \sum_{j=1}^{4} \eta_{j,i}^{(k)}  \widetilde{g}_j^{(k)} \in \Sigma_{2r}^2({\xi},\widetilde{\xi}),\ 
    \forall \indexk,\, \indexi,\,  j{\in} [1{:}4], \\ & \deg\big(\sigma_{\ell,i}^{(k)}h_\ell \big) \leq 2r, \quad \deg\big(\eta_{j,i}^{(k)}\widetilde{g}_{j}^{(k)}\big) \le 2r, \quad \forall \indexk, \, j{\in}[1{:}4], \, \indexi,
\end{align*}}where $\psi (\xi, \widetilde{\xi}) = \widetilde{\xi}$, $\varphi_i (\xi, \widetilde{\xi}) = (\xi - \widehat{\xi}_i)^2$, $\widetilde{g}_{j}^{(k)}({\xi},\widetilde{\xi}) = \widetilde{\xi} - g_j^{(k)}({\xi})$,  $j{\in}[1{:}2]$, $\widetilde{g}_{3}^{(k)}({\xi},\widetilde{\xi})  = \tau^{(k)}_1- \widetilde{\xi}$, and $\widetilde{g}_{4}^{(k)}({\xi},\widetilde{\xi}) =\widetilde{\xi} - \tau^{(k)}_2$.

\begin{proposition}[\textbf{Revenue estimation}]
\label{prop:RM}

Consider the expectation problem \eqref{problem:revenue_uq1}. Let $\Xi = \{\xi \in \R : h_1 (\xi) :=  \xi \geq 0, h_2 (\xi) := R- \xi \geq 0 \}$ be the support set, for some $R>0$; and let $\varepsilon>0$. Let $g_j^{(k)} \! \in \R[\xi]$, $j{\in}[1{:}2], \indexk$ be as defined in \eqref{eqn:revmax-g}. Then, $\min \eqref{problem:revmax-hierarchy} \geq \min ({\rm RD}_{r+1}) \geq \max \eqref{problem:revenue_uq1}$, for $r  = 2,3,4,\ldots$. Moreover, $\displaystyle \lim_{r \to \infty} \min {\eqref{problem:revmax-hierarchy}} = \max \eqref{problem:revenue_uq1}$.
\end{proposition}
\begin{proof} It is easy to check that the Archimedean condition for $\{h_1, h_2\}$ is satisfied by choosing $\overline{R}> R^2$, $\sigma_0 (\xi) = \overline{R}-R^2$, $\sigma_1 (\xi) = \frac{(R-\xi)^2}{R}$, $\sigma_2 (\xi) = \frac{\xi^2}{R} + R$, so $\overline{R} - \xi^2 = \sigma_0 + \sigma_1 h_1 + \sigma_2 h_2$. Hence, the conclusion follows directly from  \Cref{thm:SOSconvergence-minimax}.
\end{proof}

\textbf{Numerical setup.} In the following numerical experiment for revenue estimation, we consider $K=3$ customers. We fix the support set by choosing $R=12$. We generate $N = 30$ data points by sampling from a normal distribution with mean $\frac{R}{2}$ and standard deviation $\frac{R}{7}$, followed by clipping the values to the interval $[0, R]$. The polynomials $g_j^{(k)} \! \in \R[\xi]$, $j{\in}[1{:}2]$, $k{\in}[1{:}3]$ in \eqref{eqn:revmax-g} are chosen by fixing: $[a_1,d_1,b_1] = \big[4, 9, \tfrac{3}{4} \big]$,  $[a_2,d_2,b_2] = \big[ \tfrac{1}{4}, 11, \tfrac{7}{2} \big]$, and $[a_3, d_3, b_3] = \big[\tfrac{1}{110}, 14, \tfrac{23}{2}\big]$. 

   \begin{figure}[H]        
        \centering         
        \begin{subfigure}{0.425\textwidth}                         
        \includegraphics[width=\linewidth]{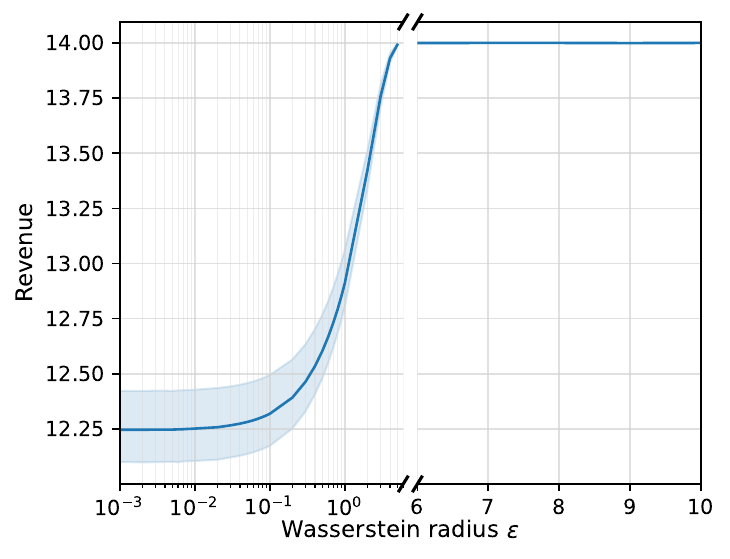}        
        \end{subfigure}        
       
        \caption{Average maximum revenue (in solid line) computed over ten independent simulation runs. The shaded region represents the inter-quantile range between the 20\textsuperscript{th} and 80\textsuperscript{th} quantiles.
        }
        \label{fig:revmax-results}
    \end{figure}

\textbf{Results and Discussion.} \Cref{fig:revmax-results} displays the mean objective value (shown in solid line) obtained by solving the hierarchy in \eqref{problem:revmax-hierarchy} with $r = 2$, as the Wasserstein radius $\varepsilon$ varies over the interval $[10^{-3}, 10]$. The objective values correspond to the maximum revenues, averaged over ten independent simulation runs. The shaded region represents the inter-quantile range between the 20\textsuperscript{th} and 80\textsuperscript{th} quantiles. The results show that a larger Wasserstein radius, which enlarges the ambiguity set, yields a higher maximum expected revenue. {Particularly, for Wasserstein radius $\varepsilon \ge 6.0$,  the maximum revenue reaches $\$14$ (up to machine precision) which coincides with the (theoretical) maximum offer price of \$14, implying that the merchant successfully targets the right customer. On this interval $[6,10]$ of values for the Wasserstein radius, the theoretical upper bound on maximum revenue is already attained (up to machine precision), indicating that further increases in the relaxation level would not yield additional improvement.

We also observe that, constrained by computational resources, \eqref{problem:revmax-hierarchy} can be solved only for $r=2$ which is the minimum relaxation order possible.

\textbf{Scalability}. We investigate the scalability and efficiency of the SDP relaxation by varying the number of customers $K$ and sample size $N$ with fixed relaxation level $r = 2$ and Wasserstein radius $\varepsilon = 10$. The first three customers are fixed and coincide with those used in the numerical setup above. For $k \ge 4$, the remaining customer parameters are generated randomly with
$a_k > 0$, $b_k \in [0,R]$, and $d_k \in [0,14]$.

\Cref{tab:objective-cpu-revmax} reports the CPU time in seconds together with corresponding maximum expected revenue. As $K$ increases, the computational cost rises sharply, exceeding 300 seconds for $K=12$ in the largest instances, while remaining increasing only mildly with the sample size $N$. \Cref{tab:problem-dimensions-revenue} highlights the growth in problem size as the number of customers $K$ increases.

Across all tested configurations, the objective values remain essentially constant and saturate at the upper bound $14$, up to negligible numerical error.

\begin{table}[ht]
\centering
\small
\begin{tabular}{@{}llrrrr@{}}
\toprule
 & & $K=3$ & $K=6$ & $K=9$ & $K=12$ \\
\midrule
CPU Time (s)
 & $N=30$  & 2.7699  & 16.8437 & 63.3640  & 202.7243 \\
 & $N=60$  & 6.4975  & 32.2161 & 110.4003 & 24.6437 \\
 & $N=90$  & 10.2844 & 51.6665 & 16.5743  & 81.4009 \\
 & $N=120$ & 16.0445 & 10.3989 & 53.2011  & 182.8085 \\
 & $N=150$ & 5.7485  & 31.3090 & 112.4932 & 316.2087 \\
\midrule
Objective Value
 & $N=30$  & 14.0000 & 14.0000 & 14.0000 & 14.0000 \\
 & $N=60$  & 13.9999 & 13.9999 & 14.0000 & 14.0000 \\
 & $N=90$  & 14.0000 & 14.0000 & 13.9999 & 14.0000 \\
 & $N=120$ & 14.0000 & 14.0000 & 14.0000 & 14.0000 \\
 & $N=150$ & 14.0000 & 13.9999 & 13.9999 & 14.0000 \\
\bottomrule
\end{tabular}
\caption{CPU times (s) and objective values for varying sample size $N$ and number of customers $K$.}
\label{tab:objective-cpu-revmax}
\end{table}

\begin{table}[ht]
\centering
\small
\begin{tabular}{@{}rrrrr@{}}
\toprule
\shortstack{$N$\\~} & \shortstack{$K$\\~} & \shortstack{Number of \\ Constraints} & \shortstack{Number of \\ Scalar Variables} & \shortstack{Number of Scalarized\\ Matrix Variables} \\
\midrule
$150$ & $3$ & $18{,}450~~~$ & $14{,}551~~~~~$ & $16{,}200~~~~~~~~$ \\
$150$ & $6$ & $36{,}900~~~$ & $28{,}951~~~~~$ & $32{,}400~~~~~~~~$ \\
$150$ & $9$ & $55{,}350~~~$ & $43{,}351$~~~~~~& $48,600~~~~~~~~$ \\
$150$ & $12$ & $73{,}800~~~$ & $57{,}751$~~~~~~& $64{,}800~~~~~~~~$ \\
\bottomrule
\end{tabular}
\caption{Problem dimensions of the resulting SDP for \eqref{problem:revmax-hierarchy} at level $r=2$, with fixed $N=1{,}000$ and increasing $K$, as reported by YALMIP.}
\label{tab:problem-dimensions-revenue}
\end{table}

% % ====================================
\subsection{Data-driven Portfolio Optimization} \label{sec:portfolio}

We now explore the use of the hierarchy in \eqref{problem:mmtd-rmin} from \Cref{section:hierarchy} for min-polynomials, involving cubic loss functions in the context of portfolio optimization. 
Moreover, we report numerical results for several relaxation levels of the hierarchy.

Consider a capital market of $m$ risky assets that can be chosen by an investor in the financial market. A portfolio is encoded by a vector of percentage weights $\boldsymbol{y} = (y_1, \ldots, y_m)^\top \in \Delta$. Assume that the uncertainty in these assets is described by the random vector $\boldsymbol{\xi} \in \R^m$ governed by a true probability distribution $\mathbb{P}$ supported on $\Xi \subset \R^m$. We further assume to have access to a training dataset $\widehat \Xi_N = \{\widehat{\boldsymbol{\xi}}_i \}_{i=1}^N$ of $N$ past realizations of the random vector $\boldsymbol{\xi}$, from which we form the empirical distribution $\widehat{\mathbb{P}} =\frac{1}{N} \Sumi \1_{\widehat{\boldsymbol{\xi}}_i}$. 

We consider the single-stage stochastic mean-risk portfolio optimization problem
\begin{equation*}
 \min_{\boldsymbol{y}\in \Delta} \bigg\{\mathbb{E}^{\mathbb{P}} \Big [ \sum_{p=1}^m y_p c_p(\boldsymbol{\xi})\Big ] + \gamma \cdot \mbox{CVaR}_{\eta}\Big[\sum_{p=1}^m y_p c_p(\boldsymbol{\xi})\Big]\bigg\}
\end{equation*}
for some continuous functions $c_p(\boldsymbol{\xi})$, $p = 1,\ldots, m$, representing the asset losses associated with the risk vector $\boldsymbol{\xi}$. This formulation minimizes a weighted combination of the mean and the conditional value-at-risk (CVaR) of the portfolio loss modeled by $\sum_{p=1}^m y_p c_p(\boldsymbol{\xi})$, with $\eta \in (0,1]$ in the CVaR-term representing the confidence level, and $\gamma \in \mathbb{R}_+$ reflects the investor's level of risk aversion.
The CVaR is a known measure for assessing tail risks, particularly in financial management \cite{rockafellar2000optimization}. For a random loss ${\rm Loss}(\boldsymbol{\xi})$, the CVaR at a pre-specified confidence level $\eta \in (0,1]$ quantifies the expected loss in the worst $\eta \times 100\%$ scenarios. It is defined by:
\begin{equation*}
\mathrm{CVaR}_\eta ({\rm Loss}(\boldsymbol{\xi})) = \inf_{\tau \in \R } \left\{ \tau + \frac{1}{\eta} \mathbb{E}^{\mathbb{P}}\left[ \max\left\{{\rm Loss}(\boldsymbol{\xi}) - \tau,\, 0 \right\} \right]  \right\},
\end{equation*}where $\mathbb{P}$ denotes the true probability distribution of the underlying uncertainty $\boldsymbol{\xi}$.

Inserting the definition of CVaR and robustifying against distributional uncertainty, the distributionally robust counterpart of this problem with respect to the Wasserstein ambiguity set $\mathbb{B}_\varepsilon (\widehat{\mathbb{P}})$ is given by
\begin{equation}   
    \label{problem:DRO-meancvarcubic}\tag{PP}
    \inf_{\boldsymbol{y} \in\Delta,\tau \in \mathbb{R}} \max \limits_{\mathbb{Q} \in \mathbb{B}_\varepsilon(\widehat{\mathbb{P}} )} \mathbb{E}^{\mathbb{Q}} \Big [ \max_{k {\in} [1{:}2]} g_1^{(k)}((\boldsymbol{y},\tau),\boldsymbol{\xi})\Big ],
\end{equation} 
where
{\small \begin{align}\label{eqn:portfolio-g}
    g_1^{(1)}((\boldsymbol{y},\tau),\boldsymbol{\xi})  = \sum_{p=1}^m y_p c_p(\boldsymbol{\xi}) + \gamma \tau 
    \quad 
    \text{and} 
    \quad 
    g_1^{(2)}((\boldsymbol{y},\tau),\boldsymbol{\xi}) =  \left(1 + \tfrac{\gamma}{\eta}\right)\sum_{p=1}^m y_p c_p(\boldsymbol{\xi}) + \left(1-\tfrac{1}{\eta}\right)\gamma\tau.
\end{align}}

We consider the ball support set $\Xi = \{\boldsymbol{\xi} \in \R^m : h_{1}(\boldsymbol{\xi}):= R^2 - \|\boldsymbol{\xi}\|_2^2  \geq 0 \}$ for some $R > 0$.

We associate \eqref{problem:DRO-meancvarcubic} with the following sequence of problems for $r \in \mathbb{N}$ with $\displaystyle 2r \ge \max\{\!\max_{p{\in}[1{:}m]}\!\deg(c_p), 2\}$:
\begin{align*}\label{problem:portfolio-hierarchy}\tag{PD$_r$}
        \min_{\substack{ \boldsymbol{y} \in \Delta, \tau \in \R\\ \lambda \geq 0, \alpha_i \in \R\\ \sigma_{i}^{(k)} \in \Sigma^2}} 
        & \lambda \varepsilon^2 - \frac{1}{N} \Sumi \alpha_i \\ 
        \st \ 
        & \ -\sum_{p=1}^m y_p c_p - \gamma \tau  + \lambda \varphi_i  - \alpha_i - \sigma_{i}^{(1)} h_1 \in \Sigma_{2r}^2(\boldsymbol{\xi}), \quad \forall \indexi,\\
        & \ -\left(1 + \tfrac{\gamma}{\eta}\right)\sum_{p=1}^m y_p c_p - \left(1-\tfrac{1}{\eta}\right)\gamma\tau  + \lambda\varphi_i  - \alpha_i - \sigma_{i}^{(2)} h_1 \in \Sigma_{2r}^2(\boldsymbol{\xi}), \ \forall \indexi,
        \\
        & \ \deg(\sigma_{i}^{(1)}) \le 2(r-1),\ \ \deg(\sigma_{i}^{(2)}) \le 2(r-1), \quad \forall \indexi,
    \end{align*}where $\varphi_i (\boldsymbol{\xi}) = \| \boldsymbol{\xi} - \widehat{\boldsymbol{\xi}}_i \|_2^2$.

For clarity, in what follows, we say $(\boldsymbol{y}^\star, \tau^\star,\mathbb{Q}^\star) \in (\Delta \times \mathbb{R} \times \mathbb{B}_{\varepsilon}(\widehat{\mathbb{P}}))$ is a solution for the min-max optimization problem \eqref{problem:DRO-meancvarcubic} if  
$$\int_\Xi \max_{k {\in} [1{:}2]} g_1^{(k)}((\boldsymbol{y}^\star,\tau^\star),\boldsymbol{\xi})\, \mathbb{Q}^\star (d\boldsymbol{\xi}) ={\rm val}\eqref{problem:DRO-meancvarcubic},$$
where ${\rm val}\eqref{problem:DRO-meancvarcubic}$ is the optimal value of  $\eqref{problem:DRO-meancvarcubic}$.  

\begin{proposition}[\textbf{Portfolio optimization}]
\label{prop:PO}
For the DRO problem \eqref{problem:DRO-meancvarcubic}, let $\Xi = \{\boldsymbol{\xi} \in \R^m :  h_1 (\boldsymbol{\xi}) := R^2 - \|\boldsymbol{\xi}\|_2^2  \geq 0 \}$ be the support set, and let $\varepsilon>0$. Let $g_1^{(1)}$ and $g_1^{(2)}$ be as defined in \eqref{eqn:portfolio-g}.  Assume that a solution $(\boldsymbol{y}^\star, \tau^\star,\mathbb{Q}^\star)$ of \eqref{problem:DRO-meancvarcubic} and minimizers $(\boldsymbol{y}_r^\star, \tau_r^\star, \lambda_r^\star, \alpha_{i,r}^\star,\sigma_{i,r}^\star)$ of \eqref{problem:portfolio-hierarchy} exist for any $r \in\mathbb{N}$. Then $\min \eqref{problem:portfolio-hierarchy} \ge \min ({\rm PD}_{r+1}) \ge {\rm val} \eqref{problem:DRO-meancvarcubic}$, for any $r \in \mathbb{N}$. Moreover, $\displaystyle \lim_{r \to \infty} \min {\eqref{problem:portfolio-hierarchy}} = {\rm val}\eqref{problem:DRO-meancvarcubic}$.
\end{proposition}

\begin{proof}
First, we note that the Archimedean condition is automatically satisfied for $\Xi$. Let $\mathcal{X}:= (\Delta \times \mathbb{R})$.  Fix $\boldsymbol{y} \in \Delta$ and $\tau \in \mathbb{R}$. Consider the worst-case expectation subproblem of \eqref{problem:DRO-meancvarcubic}:
\begin{align}\label{subproblem:DRO-meancvarcubic}\tag{${\rm PP}^{(\boldsymbol{y},\tau)}$}
\max_{\mathbb{Q} \in \mathbb{B}_\varepsilon(\widehat{\mathbb{P}} )} \mathbb{E}^{\mathbb{Q}} \Big [  \max_{k {\in} [1{:}2]} g_1^{(k)}((\boldsymbol{y},\tau),\boldsymbol{\xi})  \Big ]. 
\end{align}
By \Cref{thm:SOSconvergence-min}, $\max\eqref{subproblem:DRO-meancvarcubic}$ can be approximated by 
\begin{align*}\label{subproblem:portfolio-hierarchy}\tag{${\rm PD}_r^{(\boldsymbol{y},\tau)}$}
     \min_{\substack{\lambda \geq 0, \alpha_i \in \R\\ \sigma_{i}^{(k)} \in \Sigma^2 }} & 
     \lambda \varepsilon^2 - \frac{1}{N} \Sumi \alpha_i \\ 
    \st \ 
    & \ - \sum_{p=1}^m y_p c_p - \gamma \tau  + \lambda\varphi_i  - \alpha_i - \sigma_{i}^{(1)} h_1 \in \Sigma_{2r}^2(\boldsymbol{\xi}), \quad \forall \indexi,\\
    & \ - \left(1 + \tfrac{\gamma}{\eta}\right)\sum_{p=1}^m y_p c_p - \left(1-\tfrac{1}{\eta}\right)\gamma\tau  + \lambda\varphi_i  - \alpha_i - \sigma_{i}^{(2)} h_1 \in \Sigma_{2r}^2(\boldsymbol{\xi}),\quad  \forall \indexi, \\
     & \ \deg(\sigma_{i}^{(1)}) \le 2(r-1),\ \ \deg(\sigma_{i}^{(2)}) \le 2(r-1), \quad \forall \indexi,
\end{align*}in the sense that  
\begin{equation}\label{eqn:inner-ineq}
\min \eqref{subproblem:portfolio-hierarchy} \ge \min ({\rm PD}_{r+1}^{(\boldsymbol{y},\tau)}) \ge \max \eqref{subproblem:DRO-meancvarcubic} \quad \text{and}\quad \displaystyle \lim_{r\to \infty} \min \eqref{subproblem:portfolio-hierarchy} = \max \eqref{subproblem:DRO-meancvarcubic}.
\end{equation}
To proceed, we define $v_r: \mathcal{X} \rightarrow \mathbb{R}$ by $v_r(\boldsymbol{y},\tau):=  \min \eqref{subproblem:portfolio-hierarchy}$ for each $(\boldsymbol{y},\tau) \in \mathcal{X}$,  and 
$v:\mathcal{X} \rightarrow \mathbb{R}$ by 
$v(\boldsymbol{y},\tau) 
    := \max \eqref{subproblem:DRO-meancvarcubic}$ for each $(\boldsymbol{y},\tau) \in \mathcal{X}$. 

    With these notation, \eqref{eqn:inner-ineq} writes as:
\begin{equation}\label{ineq:vr}
v_r(\boldsymbol{y},\tau) \ge v_{r+1}(\boldsymbol{y},\tau) \ge v(\boldsymbol{y},\tau) \quad \text {and} \quad \lim_{r \to \infty} v_r(\boldsymbol{y},\tau) = v(\boldsymbol{y},\tau).
\end{equation}
For convenience, we denote 
\begin{equation}\label{eqn:pdef}
p_r := \min {\eqref{problem:portfolio-hierarchy}} =  \min_{(\boldsymbol{y},\tau) \in \mathcal{X}} v_r(\boldsymbol{y},\tau) , \quad \text{and} \quad p^\star := {\rm val}\eqref{problem:DRO-meancvarcubic} =\min_{(\boldsymbol{y},\tau)\in \mathcal{X}}  v(\boldsymbol{y},\tau).
\end{equation}
From our assumption of the existence of solutions, one sees that 
\[
p_r  =  v_r(\boldsymbol{y}_r^\star, \tau_r^\star) \quad
\mbox{and} \quad p^\star = \mathbb{E}^{\mathbb{Q}^\star} \Big [  \max_{k {\in} [1{:}2]} g_1^{(k)}((\boldsymbol{y}^\star,\tau^\star),\boldsymbol{\xi})\Big ] = \max({\rm PP}^{(\boldsymbol{y}^\star,\tau^\star)}) =v(\boldsymbol{y}^\star, \tau^\star), 
\]
where $(\boldsymbol{y}_r^\star, \tau_r^\star) \in \mathcal{X}$ and $(\boldsymbol{y}^\star, \tau^\star) \in \mathcal{X}$ are given as in the assumptions.

Firstly, we claim that $p_r \ge p_{r+1}  \ge p^\star$. Indeed, by repeatedly using \eqref{ineq:vr} and \eqref{eqn:pdef}, we obtain
\begin{align*}
p_r &= v_r(\boldsymbol{y}_r^\star,\tau_r^\star) \ge  v_{r+1}(\boldsymbol{y}_r^\star,\tau_r^\star)\ge \min_{(\boldsymbol{y},\tau) \in \mathcal{X}} v_{r+1}(\boldsymbol{y},\tau) = v_{r+1}(\boldsymbol{y}_{r+1}^\star,\tau_{r+1}^{*}) 
=  p_{r+1} ,
\end{align*}and $p_{r+1} \ge v(\boldsymbol{y}_{r+1}^\star,\tau_{r+1}^\star) \ge \min_{(\boldsymbol{y},\tau) \in \mathcal{X}} v(\boldsymbol{y},\tau) = p^\star$. 

So, $\{p_r\}$ is a non-increasing sequence which is bounded below (by $p^\ast$). In other words, $\min \eqref{problem:portfolio-hierarchy} \ge \min ({\rm PD}_{r+1}) \ge {\rm val} \eqref{problem:DRO-meancvarcubic}$. 

Let $\delta >0$ be arbitrary. Since, by \eqref{ineq:vr}, $\displaystyle \lim_{r \to \infty} v_r(\boldsymbol{y}^\star,\tau^\star) = v(\boldsymbol{y}^\star,\tau^\star)$, then there exists $\widetilde{R} >0$ such that for all $r \ge \widetilde{R}$, we have  
\[v_r(\boldsymbol{y}^\star,\tau^\star) < v(\boldsymbol{y}^\star,\tau^\star) + \delta \overset{\eqref{eqn:pdef}}= p^\star + \delta
\quad \implies \quad p_r \overset{\eqref{eqn:pdef}} = \min_{(\boldsymbol{y},\tau) \in \mathcal{X}} v_r(\boldsymbol{y},\tau) \le v_r(\boldsymbol{y}^\ast,\tau^\star) < p^\star + \delta. \]
Noting that $p_r \ge p^\star$, then for any $r\ge \widetilde{R}$, we also have $p_r - p^\star = |p_r - p^\star | < \delta.$ Since $\delta$ is arbitrary, we conclude that $\lim_{r \to \infty} p_r = p^\star$. Thus, $\displaystyle \lim_{r \to \infty} \min {\eqref{problem:portfolio-hierarchy}} = {\rm val}  \eqref{problem:DRO-meancvarcubic}$, and the asymptotic convergence holds.
\end{proof}

\textbf{Numerical setup.} In the following numerical experiments for portfolio optimization, we consider $m=3$ stocks. The support set is fixed by choosing $R=1$. The $N = 30$ data points are drawn uniformly from the unit sphere. The polynomials $g_1^{(k)} \! \in \R[\xi]$, $k{\in}[1{:}2]$, in \eqref{eqn:portfolio-g} are set with $\gamma = 10$, $\eta = 0.2$, $c_1(\boldsymbol{\xi})  = -1 + \xi_1 + \xi_1\xi_2 -\xi_1\xi_3 - 2\xi_1^3$,  $c_2(\boldsymbol{\xi}) = -1 - \xi_1\xi_2 + \xi_2^2 -\xi_2\xi_3 + \xi_2^3$, and $c_3(\boldsymbol{\xi})  = -1 + \xi_2\xi_3 - \xi_3^2 -\xi_3^3$, where $c_1,c_2,c_3$ are consistent with \cite[Example~5.5]{nie2023distributionally}.

    \begin{figure}[H]        
        \centering         
        \begin{subfigure}{0.40\textwidth}                         
            \includegraphics[width=\linewidth]{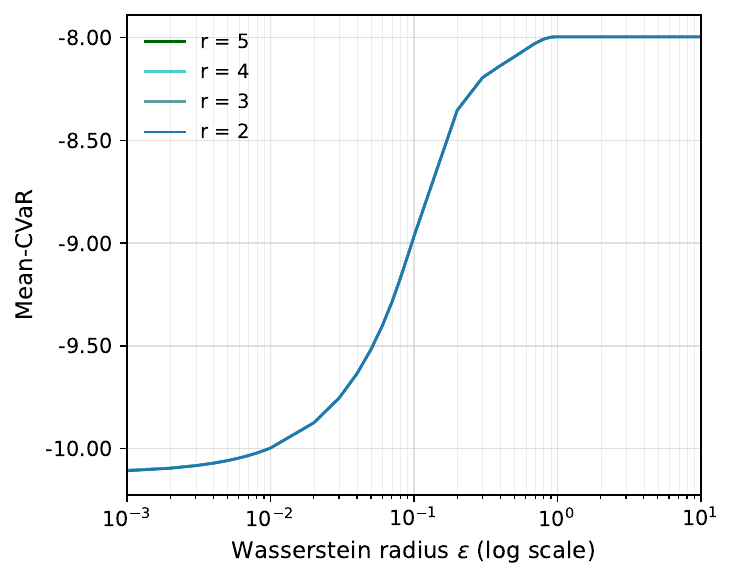}  
        \end{subfigure}      
    \caption{Average in-sample mean-CVaR risk calculated over ten independent simulation runs, as the Wasserstein radius $\varepsilon$ varies within the interval $[10^{-3}, 10]$.}
    \label{fig:portfolio-results}
    \end{figure}

\textbf{Results and Discussion.} \Cref{fig:portfolio-results} displays the mean objective values obtained by solving the hierarchy in \eqref{problem:portfolio-hierarchy} over varying Wasserstein radius $\varepsilon \in [10^{-3}, 10]$ for different levels $r \in \{2,3,4,5\}$. The objective values correspond to the mean-CVaR risks, and are averaged over ten independent simulation runs. The curves flatten once the Wasserstein radius reaches a certain threshold ($\varepsilon \approx 1$), indicating that higher uncertainty beyond a certain threshold does not lead to worse investment losses and risks. The results show that the gains from increasing the relaxation level are negligible. We note that \eqref{problem:portfolio-hierarchy} can still be solved for higher relaxation levels $r > 5$, although this comes at the expense of longer computation times.

\textbf{CPU time.} On average, a single instance of \eqref{problem:portfolio-hierarchy} was solved in 0.4261 seconds for $r=2$, 1.0180 seconds for $r=3$, 2.2799 seconds for $r=4$, and 6.4355 seconds for $r=5$.

\textbf{Scalability}. Henceforth, we fix the relaxation level at $r=2$ and the Wasserstein radius at $\varepsilon = 10.0$, and vary both the number $m$ of stocks and the number $N$ of data points. In this manner, the dimension of the result optimization problem increases in two complementary directions: increasing $m$ enlarges the decision space of the portfolio weights and the underlying polynomial model, while increasing $N$ raises the number of sample-dependent constraints in the Wasserstein ambiguity set. 

To extend the numerical setup for $m =3$ to cases where $m>3$, we fix the cost functions for the additional stocks by setting $c_p(\boldsymbol{\xi}) = -\xi_p^3$ for $p=4,\ldots,m$, while retaining $c_1$, $c_2$, and $c_3$ from the baseline setup. Specifically, we now consider $m=3,6,9,12$ stocks. In varying $N$, we consider $N=30, 60, 90, 120, 150$, corresponding to increasing the sample sizes in the empirical measure used to define the Wasserstein ambiguity set, which may be interpreted in the portfolio context as increasing the number of observed trading days.

\begin{table}[ht]
\centering
\small
\begin{tabular}{@{}lrrrr@{}}
\toprule
    & $m=3$ & $m=6$ & $m=9$ & $m=12$ \\
\midrule
  $N=30$  & $0.7131$ & $1.5564$ & $5.7414$  & $28.2044$  \\
  $N=60$  & $1.2574$ & $3.0730$ & $17.2813$ & $65.2100$  \\
  $N=90$  & $1.3550$ & $5.1252$ & $22.2863$ & $112.8436$ \\
  $N=120$ & $4.1532$ & $8.8236$ & $32.7673$ & $192.6790$ \\
  $N=150$ & $2.9811$ & $13.9251$ & $57.2876$ & $315.6149$ \\
% \midrule
% Objective Value
%  & $N=30$  & $-7.9959$ & $-7.9959$ & $-7.9959$ & $-7.9959$ \\
%  & $N=60$  & $-7.9959$ & $-7.9959$ & $-7.9959$ & $-7.9959$ \\
%  & $N=90$  & $-7.9959$ & $-7.9959$ & $-7.9959$ & $-7.9959$ \\
%  & $N=120$ & $-7.9959$ & $-7.9959$ & $-7.9959$ & $-7.9959$ \\
%  & $N=150$ & $-7.9959$ & $-7.9959$ & $-7.9959$ & $-7.9959$ \\
\bottomrule
\end{tabular}
\caption{CPU times (seconds) for varying $N$ and $m$.} \label{tab:objective-cpu}
\end{table}

\begin{table}[ht]
\centering
\small
\begin{tabular}{@{}crrrrr@{}}
\toprule
\shortstack{$N$\\~} & \shortstack{$m$\\~} & \shortstack{Number of \\ Constraints } & \shortstack{Number of \\ Scalar Variables} & \shortstack{Number of \\ Scalarized Matrix Variables} \\
\midrule
$150$ & $3$  & $13{,}501~~$ & $3{,}155~~~~~~$   & $19{,}500~~~~~~~~~~~~$   \\
$150$ & $6$  & $71{,}401~~$ & $8{,}558~~~~~~$   & $130{,}200~~~~~~~~~~~~$   \\
$150$ & $9$  & $231{,}001~~$ & $16{,}661~~~~~~$  & $478{,}500~~~~~~~~~~~~$  \\
$150$ & $12$ & $573{,}301~~$ & $27{,}464~~~~~~$  & $1{,}283{,}100~~~~~~~~~~~~$  \\
\bottomrule
\end{tabular}
\caption{Problem dimensions of the resulting SDP for \eqref{problem:portfolio-hierarchy} at level $r=2$, with fixed $N=150$ and increasing $m$, as reported by YALMIP.}
\label{tab:problem-dimensions}
\end{table}

The results in \Cref{tab:objective-cpu} show that the CPU time increases rapidly with the number of assets $m$, while the growth with respect to the number of samples $N$ is comparatively moderate. This behavior is typical for polynomial optimization problems solved via SOS relaxations, as increasing $m$ induces a combinatorial growth in the number of polynomial monomials and, consequently, in the size of the semi-definite matrix variables and the number of constraints. As illustrated in \Cref{tab:problem-dimensions}, the number of scalarized semi-definite variables grows by more than an order of magnitude as $m$ increases, which substantially raises the per-iteration cost of the interior-point method employed by MOSEK. In contrast, increasing $N$ mainly introduces additional sample-dependent constraints and leads to a comparatively milder increase in computational effort. 

\subsection{Strengths and Limitations of the Computational Study}\label{sec:strength-limit}

A key strength of our approach is that the SDP relaxations enable the computation of optimal values of these otherwise infinite-dimensional problems with piecewise polynomials that are prevalent in several decision-making models in diverse disciplines. The results confirm that the hierarchy can produce solutions within a reasonable time for problem instances of moderate dimensionality and relaxation degrees.

However, a notable limitation is computational scalability. Depending on the problem’s dimensionality, available computational resources may restrict solving the SDP hierarchy beyond a certain level of relaxation (e.g., the revenue estimation in \Cref{sec:rev_max} can only be solved with relaxation level $r=2$), thereby limiting the approximation quality or the expressiveness of the model in higher dimensions. This is due to the size of the resulting SDPs, which grows rapidly with the degree and the number of variables. 

%%%%%%%%%%%%%%%%%%%%%%%%%%%%%%%%%%%%%%%%%%%
\section{Conclusion and Outlook} \label{sec:conclusion}

In this paper, we developed convergent semi-definite program hierarchies for a class of Wasserstein optimization problems with piecewise polynomials. These hierarchies enable the computation or estimation of the optimal values of otherwise infinite-dimensional DRO problems through SDPs, efficiently solvable by off-the-shelf software. We established conditions under which the hierarchies converge asymptotically and finitely.

The construction of these hierarchies relied on techniques from algebraic geometry, especially the sum-of-squares representations of positivity as well as non-negativity, and convex analysis tools such as convex duality and the minimax theorem. These tools allowed for the reformulation of the hard semi-infinite constraints into sum-of-squares constraints, expressible as linear matrix inequalities. 

We demonstrated the computational feasibility of our approach
through concrete examples in revenue estimation and portfolio optimization. Future research could explore: 
\begin{itemize}

    \item Moment ambiguity-based approach for hierarchies of  piecewise polynomial functions, extending the piecewise SOS-convex case in \cite{huang2024piecewise}.
    
\item Bounded-degree SDP hierarchies for broad classes of DRO problems by using recent advances  \cite{chuong2019new,weisser2018sparse}. 

\item Other hierarchies based on scaled-diagonally-dominant-sum-of-squares \cite{ahmadi2019dsos}, aimed at enhancing scalability.

\item The recovery of optimal solutions of expectation problems \eqref{problem:mmt-minimax} from the solutions of associated SDPs.
\end{itemize}

%%%%%%%%%%%%%%%%%%%%%%%%%%%%%%%%%%%%
\begin{appendices}
\section{Technical Details} \label{appendix:finite-reduction}

This appendix provides technical proof.

\begin{proof}[Proof of \Cref{lemma:semi-inf}]

Recall that $\Xi$ is compact and $\mathcal{M} (\Xi)$ refers to the space of (finite signed) regular Borel measures on $\Xi$ and $\mathcal{P} (\Xi) \subseteq \mathcal{M} (\Xi)$ is the set of probability measures. Let $C(\Xi)$ be the space of continuous functions supported on $\Xi$
equipped with the supremum norm. We now equip the $\mathcal{M}(\Xi)$ with the weak$^*$ topology. 

By the law of total probability \cite[see, e.g.,][Theorem 4.2]{mohajerin2018data},  
{\small \begin{align}\label{convex}
    \min \eqref{problem:mmt_0} &=\inf_{\substack{\mathbb{Q}_i \in \mathcal{P}(\Xi) }} \left \{\frac{1}{N}  \Sumi  \mathbb{E}^{\mathbb{Q}_i} [ g(\boldsymbol{\xi}) ]: \frac{1}{N}  \Sumi  \mathbb{E}^{\mathbb{Q}_i} [\|\boldsymbol{\xi}- \widehat{\boldsymbol{\xi}}_i \|_2^2] \le \varepsilon^2 \right \} \nonumber \\ &= \inf_{\mathbb{Q}_i \in \mathcal{M} (\Xi) } \left \{\frac{1}{N}  \Sumi  \mathbb{E}^{\mathbb{Q}_i} [g(\boldsymbol{\xi})  ]+ \Sumi \chi_{\mathcal{P} (\Xi)} (\mathbb{Q}_i): \frac{1}{N}  \Sumi  \mathbb{E}^{\mathbb{Q}_i} [\|\boldsymbol{\xi}- \widehat{\boldsymbol{\xi}}_i \|_2^2 ] \le \varepsilon^2 \right \}. 
\end{align} }Here, $\chi_{\mathcal{P} (\Xi)} (\mathbb{Q})$ is the indicator function which takes value $0$ if $\mathbb{Q} \in \mathcal{P} (\Xi)$ and $+\infty$ otherwise.

Since $\mathcal{P} (\Xi)$ is a weakly$^*$ compact and convex subset of $\mathcal{M} (\Xi)$ \cite[Example~2.122]{bonnans2013perturbation}, and $\mathcal{P} (\Xi)$ is non-empty due to $\Xi \neq \emptyset$, so the indicator function $\chi_{\mathcal{P}(\Xi)}$ is a proper convex function. A direct verification shows that the problem of the right-hand side of \eqref{convex} is a convex optimization problem on $\prod_{i=1}^N \mathcal{M}(\Xi)$ with a proper convex objective function. 

Note that Slater's condition holds for the convex problem of the right-hand side of \eqref{convex}, i.e., the existence of $\mathbb{Q}_i \in \mathcal{P} (\Xi)$, $\indexi$, such that $\frac{1}{N} \Sumi \mathbb{E}^{\mathbb{Q}_i} [\| \boldsymbol{\xi} - \widehat{\boldsymbol{\xi}}_i \|_2^2 ] < \varepsilon^2$. This can be seen by taking $\mathbb{Q}_i = \1_{\widehat{\boldsymbol{\xi}}_i}$ which gives $\frac{1}{N} \Sumi \mathbb{E}^{\1_{\widehat{\boldsymbol{\xi}}_i}} [\| \boldsymbol{\xi} - \widehat{\boldsymbol{\xi}}_i \|_2^2]  = \frac{1}{N} \Sumi \| \widehat{\boldsymbol{\xi}}_i - \widehat{\boldsymbol{\xi}}_i \|_2^2 = 0 < \varepsilon^2$.

Now, applying the convex duality theorem \cite[Theorem~2.9.3]{zalinescu2002convex},
we have \begin{align*}
    \min \eqref{problem:mmt_0} &= \max_{\lambda \geq 0} \inf_{\substack{\mathbb{Q}_i \in \mathcal{P}(\Xi) }} \left \{ \frac{1}{N} \Sumi \mathbb{E}^{\mathbb{Q}_i} [g (\boldsymbol{\xi})] + \lambda \Big (\frac{1}{N} \Sumi \mathbb{E}^{\mathbb{Q}_i} [\|\boldsymbol{\xi} - \widehat{\boldsymbol{\xi}}_i \|_2 ^2] - \varepsilon^2 \Big ) \right \}  \\ &= \max_{\lambda \geq 0} \left \{- \lambda \varepsilon^2 + \frac{1}{N} \Sumi \inf_{\mathbb{Q}_i \in \mathcal{P} (\Xi)} \mathbb{E}^{\mathbb{Q}_i} [g(\boldsymbol{\xi}) + \lambda \| \boldsymbol{\xi} - \widehat{\boldsymbol{\xi}}_i \|_2^2]  \right \}. 
\end{align*}

Next, we claim that, for each $\indexi$, \begin{align*}
    \inf_{\mathbb{Q}_i \in \mathcal{P} (\Xi)} \int_\Xi \Big (g (\boldsymbol{\xi}) + \lambda \| \boldsymbol{\xi} - \widehat{\boldsymbol{\xi}}_i \|_2^2 \Big ) \mathbb{Q}_i (d\boldsymbol{\xi}) = \inf_{\boldsymbol{\xi} \in \Xi} \Big ( g(\boldsymbol{\xi}) +  \lambda \|\boldsymbol{\xi} - \widehat{\boldsymbol{\xi}}_i \|_2^2 \Big ). 
\end{align*}

To see this, we note that the set of Dirac measures is a subset of $\mathcal{P} (\Xi)$ since $\1_{\boldsymbol{\xi}} \in \mathcal{P}(\Xi)$ for any $\boldsymbol{\xi} \in \Xi$, and so, $\displaystyle \inf_{\mathbb{Q} \in \mathcal{P} (\Xi)} \int_\Xi f(\boldsymbol{\xi})\ \mathbb{Q} (d\boldsymbol{\xi}) \leq \inf_{\boldsymbol{\xi} \in \Xi} f(\boldsymbol{\xi})$ for any continuous function $f$ on $\R^m$. The reverse inequality follows from the fact that $\displaystyle \int_\Xi f(\boldsymbol{\xi})\ \mathbb{Q} (d\boldsymbol{\xi}) \geq \int_\Xi \Big (\inf_{\boldsymbol{\xi}' \in \Xi} f(\boldsymbol{\xi}') \Big ) \mathbb{Q} (d\boldsymbol{\xi}) = \inf_{\boldsymbol{\xi}' \in \Xi} f(\boldsymbol{\xi}')$ for $\mathbb{Q} \in \mathcal{P} (\Xi)$ and for any continuous function $f$. Thus, the claim follows.

Therefore, by the claim, we see that \begin{align*}
    \min \eqref{problem:mmt_0} &= \max_{\lambda \geq 0}\left \{ -\lambda \varepsilon^2 + \frac{1}{N} \Sumi \inf_{\boldsymbol{\xi} \in \Xi} \Big (g (\boldsymbol{\xi}) + \lambda \| \boldsymbol{\xi} - \widehat{\boldsymbol{\xi}}_i \|_2^2 \Big ) \right \} \\ &= \max_{\substack{\lambda \geq 0, \alpha_i \in \R}} \left \{ -\lambda \varepsilon^2 + \frac{1}{N} \Sumi \alpha_i : \inf_{\boldsymbol{\xi} \in \Xi} \Big (g (\boldsymbol{\xi}) + \lambda \| \boldsymbol{\xi} - \widehat{\boldsymbol{\xi}}_i \|_2^2 - \alpha_i \Big ) \geq 0,\ \forall \indexi \right \}.
\end{align*}
Hence, $\min\eqref{problem:mmt_0} = \max \eqref{problem:mmtd}$. 
\end{proof}

\end{appendices}

%%%%%%%%%%%%%%%%%%%%%%%%%%%%%%%%%%%%%%%%%%%%%
\section*{Declarations}

\textbf{Data Statement}. 
All data generated and analyzed in this study are provided within the article. The data used in the numerical experiments were generated randomly, and we have clearly described the procedure for reproducing them.

\textbf{Conflict of interest statement}. The last two authors declare that they are Associate Editors for the Journal of Optimization
Theory and Applications. The other three authors have no conflict of interest to declare that are relevant to
the content of this article.

\textbf{Acknowledgment}. The authors are grateful to the referees for their helpful comments and valuable suggestions which have contributed to the final preparation of the paper.

\setlength{\bibsep}{0pt plus 0.3ex}

{
\renewcommand{\emph}[1]{#1}
\let\em\relax
\bibliographystyle{abbrv}
\bibliography{sample}

@article{gerber1998utility,
  title={Utility functions: From risk theory to finance},
  author={Gerber, Hans U and Pafum, G{\'e}rard},
  journal={North American Actuarial Journal},
  volume={2},
  number={3},
  pages={74--91},
  year={1998},
  publisher={Taylor \& Francis},
  doi = {https://doi.org/10.1080/10920277.1998.10595728}
}

@article{benishay1987fourth,
  title={A Fourth-Degree Polynomial Utility Function and Its Implications for Investors' Responses toward Four Moments of the Wealth Distribution},
  author={Benishay, Haskel},
  journal={Journal of Accounting, Auditing \& Finance},
  volume={2},
  number={3},
  pages={203--228},
  year={1987},
  publisher={SAGE Publications Sage CA: Los Angeles, CA}
}

@article{ho2023adversarial,
title={Adversarial classification via distributional robustness with {W}asserstein ambiguity},
  author={Ho-Nguyen, Nam and Wright, Stephen J},
  journal={Mathematical Programming},
  volume={198},
  number={2},
  pages={1411--1447},
  year={2023},
  publisher={Springer}
}

@article{lasserre2008semidefinite,
  title={A semidefinite programming approach to the generalized problem of moments},
  author={Lasserre, Jean B},
  journal={Mathematical Programming},
  volume={112},
  pages={65--92},
  year={2008},
  publisher={Springer}
}

@article{kuhn2024dro,
  title={Distributionally {R}obust {O}ptimization},
  author={Kuhn, Daniel and Shafiee, Soroosh and Wiesemann, Wolfram},
  journal={Acta Numerica},
  volume={34},
  pages={579--804},
  year={2025},
  publisher={Cambridge University Press}
}

@article{putinar1993positive,
  title={Positive polynomials on compact semi-algebraic sets},
  author={Putinar, Mihai},
  journal={Indiana University Mathematics Journal},
  volume={42},
  number={3},
  pages={969--984},
  year={1993},
  publisher={JSTOR}
}

@article{laraki2012semidefinite,
  title={Semidefinite programming for min--max problems and games},
  author={Laraki, Rida and Lasserre, Jean B},
  journal={Mathematical Programming},
  volume={131},
  pages={305--332},
  year={2012},
  publisher={Springer}
}

@article{huang2025convexifiable,
  title={Convexifiable quadratic inequality systems: New minimax {S}-lemma and exact {SOCP}s for classes of distributionally robust optimization problems},
  author={Huang, QY and Jeyakumar, V and Li, G and Huyen, DTK},
  journal={Journal of Global Optimization},
  pages={1--34},
  year={2025},
  publisher={Springer}
}

@article{bach2025sum,
  title={Sum-of-squares relaxations for polynomial min--max problems over simple sets},
  author={Bach, Francis},
  journal={Mathematical Programming},
  volume={209},
  number={1},
  pages={475--501},
  year={2025},
  publisher={Springer}
}

@book{pflug2014multistage,
  title={Multistage Stochastic Optimization},
  author={Pflug, Georg Ch and Pichler, Alois},
  volume={1104},
  year={2014},
  publisher={Springer}
}

@article{blanchet2019quantifying,
  title={Quantifying distributional model risk via optimal transport},
  author={Blanchet, Jose and Murthy, Karthyek},
  journal={Mathematics of Operations Research},
  volume={44},
  number={2},
  pages={565--600},
  year={2019},
  publisher={INFORMS}
}

@article{gao2023distributionally,
  title={Distributionally robust stochastic optimization with {W}asserstein distance},
  author={Gao, Rui and Kleywegt, Anton},
  journal={Mathematics of Operations Research},
  volume={48},
  number={2},
  pages={603--655},
  year={2023},
  publisher={INFORMS}
}

@article{huang2024piecewise,
  title={Piecewise {SOS}-convex moment optimization and applications via exact semi-definite programs},
  author={Huang, QY and Jeyakumar, V and Li, G},
  journal={EURO Journal on Computational Optimization},
  pages={100094},
  year={2024},
  publisher={Elsevier}
}

@article{klerk2020distributionally,
  title={Distributionally robust optimization with polynomial densities: Theory, models and algorithms},
  author={de Klerk, Etienne and Kuhn, Daniel and Postek, Krzysztof},
  journal={Mathematical Programming},
  volume={181},
  pages={265--296},
  year={2020},
  publisher={Springer}
}

@article{postek2016computationally,
  title={Computationally tractable counterparts of distributionally robust constraints on risk measures},
  author={Postek, Krzysztof and den Hertog, Dick and Melenberg, Bertrand},
  journal={SIAM Review},
  volume={58},
  number={4},
  pages={603--650},
  year={2016},
  publisher={SIAM}
}

@article{van2019distributionally,
  title={Distributionally robust expectation inequalities for structured distributions},
  author={Van Parys, Bart PG and Goulart, Paul J and Morari, Manfred},
  journal={Mathematical Programming},
  volume={173},
  pages={251--280},
  year={2019},
  publisher={Springer}
}

@article{lasserre2009convexity,
  title={Convexity in semialgebraic geometry and polynomial optimization},
  author={Lasserre, Jean B},
  journal={SIAM Journal on Optimization},
  volume={19},
  number={4},
  pages={1995--2014},
  year={2009},
  publisher={SIAM}
}

@book{ben2009robust,
  title={Robust Optimization},
  author={Ben-Tal, Aharon and El Ghaoui, Laurent and Nemirovski, Arkadi},
  volume={28},
  year={2009},
  publisher={Princeton University Press},
  address={Princeton},
}

@book{bonnans2013perturbation,
  title="Perturbation Analysis of Optimization Problems",
  author="Bonnans, J. Fr{\'e}d{\'e}ric and Shapiro, Alexander",
  year="2013",
  publisher="Springer Science \& Business Media",
  address="New York",
}

@article{zhen2023unified,
  title={A Unified Theory of Robust and Distributionally Robust Optimization via the Primal-Worst-Equals-Dual-Best Principle},
  author={Zhen, Jianzhe and Kuhn, Daniel and Wiesemann, Wolfram},
  journal={Operations Research},
  year={2023},
  publisher={INFORMS}
}

@article{mohajerin2018data,
  title={Data-driven distributionally robust optimization using the {W}asserstein metric: Performance guarantees and tractable reformulations},
  author={Mohajerin Esfahani, Peyman and Kuhn, Daniel},
  journal={Mathematical Programming},
  volume={171},
  number={1},
  pages={115--166},
  year={2018},
  publisher={Springer}
}

@article{delage2010distributionally,
  title="Distributionally robust optimization under moment uncertainty with application to data-driven problems",
  author="Delage, Erick and Ye, Yinyu",
  journal="Operations Research",
  volume="58",
  number="3",
  pages="595--612",
  year="2010",
  publisher="INFORMS",
}

@article{wiesemann2014distributionally,
  title="Distributionally robust convex optimization",
  author="Wiesemann, Wolfram and Kuhn, Daniel and Sim, Melvyn",
  journal="Operations Research",
  volume="62",
  number="6",
  pages="1358--1376",
  year="2014",
  publisher="INFORMS",
}

@article{nie2023distributionally,
  title="Distributionally robust optimization with moment ambiguity sets",
  author="Nie, Jiawang and Yang, Liu and Zhong, Suhan and Zhou, Guangming",
  journal="Journal of Scientific Computing",
  volume="94",
  number="1",
  pages="12",
  year="2023",
  publisher="Springer",
}

@article{ahmadi2013complete,
  title="A complete characterization of the gap between convexity and {SOS}-convexity",
  author="Ahmadi, Amir Ali and Parrilo, Pablo A.",
  journal="SIAM Journal on Optimization",
  volume="23",
  number="2",
  pages="811--833",
  year="2013",
  publisher="SIAM",
}

@article{helton2010semidefinite,
  title="Semidefinite representation of convex sets",
  author="Helton, J William and Nie, Jiawang",
  journal="Mathematical Programming",
  volume="122",
  pages="21--64",
  year="2010",
  publisher="Springer",
}

@article {wozabal2012framework,
    AUTHOR = {Wozabal, David},
     TITLE = {A framework for optimization under ambiguity},
   JOURNAL = {Annals of Operations Research},
    VOLUME = {193},
      YEAR = {2012},
     PAGES = {21--47},
      ISSN = {0254-5330,1572-9338},
   MRCLASS = {90C15 (90C22)},
  MRNUMBER = {2874755},
MRREVIEWER = {Uday\ V.\ Shanbhag},
       DOI = {10.1007/s10479-010-0812-0},
       URL = {\url{https://doi.org/10.1007/s10479-010-0812-0}},
}

@book{zalinescu2002convex,
  title="Convex Analysis in General Vector Spaces",
  author="Zalinescu, Constantin",
  year="2002",
  publisher="World Scientific",
  address="Singapore",
}

@article{rockafellar2000optimization,
  title="Optimization of conditional value-at-risk",
  author="Rockafellar, R. Tyrrell and Uryasev, Stanislav",
  journal="Journal of Risk",
  volume="2",
  pages="21--42",
  year="2000",
  publisher="Citeseer",
}

@article{han2015convex,
  title={Convex optimal uncertainty quantification},
  author={Han, Shuo and Tao, Molei and Topcu, Ufuk and Owhadi, Houman and Murray, Richard M},
  journal={SIAM Journal on Optimization},
  volume={25},
  number={3},
  pages={1368--1387},
  year={2015},
  publisher={SIAM}
}

@article{popescu2005semidefinite,
  title={A semidefinite programming approach to optimal-moment bounds for convex classes of distributions},
  author={Popescu, Ioana},
  journal={Mathematics of Operations Research},
  volume={30},
  number={3},
  pages={632--657},
  year={2005},
  publisher={INFORMS}
}

@incollection{kuhn2019wasserstein,
  title={{W}sserstein distributionally robust optimization: Theory and applications in machine learning},
  author={Kuhn, Daniel and Mohajerin Esfahani, Peyman and Nguyen, Viet Anh and Shafieezadeh-Abadeh, Soroosh},
  booktitle={Operations Research \& Management Science in the Age of Analytics},
  pages={130--166},
  year={2019},
  publisher={Informs}
}

@inproceedings{klerk2019survey,
  title={A survey of semidefinite programming approaches to the generalized problem of moments and their error analysis},
  author={de Klerk, Etienne and Laurent, Monique},
  booktitle={World Women in Mathematics 2018 },
  pages={17--56},
  year={2019},
  organization={Springer},
  address={Switzerland},
  DOI = {10.1007/978-3-030-21170-7\_1},
}

@book{lasserre2009moments,
  title={Moments, Positive Polynomials and Their Applications},
  author={Lasserre, Jean Bernard},
  volume={1},
  year={2009},
  publisher={World Scientific},
  address={London},
}

@book{lasserre2015introduction,
  title={An Introduction to Polynomial and Semi-Algebraic Optimization},
  author={Lasserre, Jean Bernard},
  volume={52},
  year={2015},
  publisher={Cambridge University Press},
  address={Cambridge},
}

@article{scheiderer2003sums,
  title={Sums of squares on real algebraic curves},
  author={Scheiderer, Claus},
  journal={Mathematische Zeitschrift},
  volume={245},
  pages={725--760},
  year={2003},
  publisher={Springer}
}

@article{de2011lasserre,
  title={On the {L}asserre hierarchy of semidefinite programming relaxations of convex polynomial optimization problems},
  author={De Klerk, Etienne and Laurent, Monique},
  journal={SIAM Journal on Optimization},
  volume={21},
  number={3},
  pages={824--832},
  year={2011},
  publisher={SIAM}
}

@article{chuong2019new,
  title={A new bounded degree hierarchy with {SOCP} relaxations for global polynomial optimization and conic convex semi-algebraic programs},
  author={Chuong, Thai Doan and Jeyakumar, V and Li, Guoyin},
  journal={Journal of Global Optimization},
  volume={75},
  number={4},
  pages={885--919},
  year={2019},
  publisher={Springer}
}

@article{weisser2018sparse,
  title={Sparse-{BSOS}: A bounded degree {SOS} hierarchy for large scale polynomial optimization with sparsity},
  author={Weisser, Tillmann and Lasserre, Jean B and Toh, Kim-Chuan},
  journal={Mathematical Programming Computation},
  volume={10},
  pages={1--32},
  year={2018},
  publisher={Springer}
}

@article{auslender2009penalty,
  title={Penalty and smoothing methods for convex semi-infinite programming},
  author={Auslender, Alfred and Goberna, Miguel A and L{\'o}pez, Marco A},
  journal={Mathematics of Operations Research},
  volume={34},
  number={2},
  pages={303--319},
  year={2009},
  publisher={INFORMS}
}

@article{goberna2002linear,
  title={Linear semi-infinite programming theory: An updated survey},
  author={Goberna, Miguel A and Lopez, Marco A},
  journal={European Journal of Operational Research},
  volume={143},
  number={2},
  pages={390--405},
  year={2002},
  publisher={Elsevier}
}

@article{jeyakumar2014convergence,
  title={Convergence of the {L}asserre hierarchy of {SDP} relaxations for convex polynomial programs without compactness},
  author={Jeyakumar, Vaithilingam and Pham, TS and Li, Guoyin},
  journal={Operations Research Letters},
  volume={42},
  number={1},
  pages={34--40},
  year={2014},
  publisher={Elsevier}
}

@article{ahmadi2019dsos,
  title={{DSOS} and {SDSOS} optimization: More tractable alternatives to sum of squares and semidefinite optimization},
  author={Ahmadi, Amir Ali and Majumdar, Anirudha},
  journal={SIAM Journal on Applied Algebra and Geometry},
  volume={3},
  number={2},
  pages={193--230},
  year={2019},
  publisher={SIAM}
}

@article{yue2022linear,
  title={On linear optimization over {W}asserstein balls},
  author={Yue, Man-Chung and Kuhn, Daniel and Wiesemann, Wolfram},
  journal={Mathematical Programming},
  volume={195},
  number={1},
  pages={1107--1122},
  year={2022},
  publisher={Springer}
}
}

\end{document}